\documentclass{amsart}
\usepackage{amsfonts}
\usepackage{amssymb}
\usepackage{times}
\usepackage{xcolor}
\usepackage{tkz-graph}
\usepackage{url}
\usepackage{float}
\usepackage{tasks}

\usepackage{cite}
\usepackage{hyperref}

\usepackage{amsfonts}
\usepackage{amsmath}
\usepackage{eurosym}
\usepackage{geometry}
\usepackage{bm}

\def\P{\bm P}
\def\GL{\mathop{\hbox{\rm GL}}}
\usepackage{caption,booktabs}

\theoremstyle{plain}
\newtheorem{theorem}{Theorem}[section]

\newtheorem{corollary} {Corollary}[section]
 
\newtheorem{definition} {Definition}[section]

\newtheorem{lemma} {Lemma}[section]

\newtheorem{remark} {Remark}[section]

\renewenvironment{proof}[1][Proof]{\noindent\textbf{#1.} }{\ \rule{0.5em}{0.5em}}

\begin{document}

\title{Degenerations of Poisson algebras}

\thanks{The second author was  supported by the Centre for Mathematics of the University of Coimbra - UIDB/00324/2020, funded by the Portuguese Government through FCT/MCTES and by the Spanish Government through the Ministry of Universities grant `Margarita Salas', funded by the European Union - NextGenerationEU}

\author[H. Abdelwahab]{Hani Abdelwahab}
\address{Hani Abdelwahab.
\newline \indent Mansoura University, Faculty of Science, Department of Mathematics (Egypt).}
\email{{\tt haniamar1985@gmail.com}}

\author[Amir Fernández Ouaridi]{Amir Fernández Ouaridi}
\address{Amir Fernández Ouaridi. \newline \indent University of Cádiz, Department of Mathematics, Puerto Real (Espa\~na).
\newline \indent University of Coimbra, CMUC, Department of Mathematics, Coimbra
(Portugal).}
\email{{\tt amir.fernandez.ouaridi@gmail.com}}

\author[Cándido Martín González]{Cándido Martín González}
\address{Cándido Martín González.
\newline \indent University of Málaga, Department of Mathematics, Málaga (Espa\~na).}
\email{{\tt candido\_m@uma.es}}


\thispagestyle{empty}

\begin{abstract}
We construct a method to obtain the algebraic classification of Poisson algebras defined on a commutative associative algebra, and we apply it to obtain the classification of the $3$-dimensional Poisson algebras. In addition, we study the geometric classification, the graph of degenerations and the closures of the orbits of the variety of $3$-dimensional Poisson algebras. Finally, we also study the algebraic classification of the Poisson algebras defined on a commutative associative null-filiform or filiform algebra and, to enrich this classification, we study the degenerations between these particular Poisson algebras.

\bigskip

{\it 2020MSC}: 17A30,
17A40,
17B30,
17D99,
14D06,
14L30.

\end{abstract}

\maketitle

\section*{{\protect\Large   } Introduction}

The algebraic and geometric classification of different varieties of low-dimensional algebras is an interesting problem on which numerous works have been published in recent years. The first consists in determining the algebras modulo isomorphisms of a given class. The second consists in determining the irreducible components of a variety of algebras with respect to the Zariski topology. For example, both problems have been studied for $2$-dimensional algebras in \cite{kv16}, $2$-dimensional pre-Lie algebras in \cite{bb09},  
 $2$-dimensional terminal algebras in \cite{cfk19},
 $3$-dimensional Novikov algebras in \cite{bb14},  
 $3$-dimensional Jordan algebras in \cite{gkp},  
 $3$-dimensional Jordan superalgebras in \cite{maria},
 $3$-dimensional Leibniz and $3$-dimensional anticommutative algebras  in \cite{ikv19},
 $4$-dimensional Lie algebras in \cite{BC99}, among many others.  
 Moreover, if we add the works in which certain subvarieties of algebras, such as the subvariety of nilpotent (or solvable) algebras of a given class is considered \cite{degs22, ale, contr11, gkk19, ikv17, kppv, kpv, kv17, S90}, the amount of papers on this topic in last years increases even more.
Nevertheless, the geometry of the varieties of algebras endowed with two multiplications is a subject still to be explored. In this work, we consider a generalization of the geometrical notions and results used for algebras with a single multiplication, established by Grunewald and O'Halloran in \cite{GRH, GRH2}, to algebras with two multiplications, applying these ideas with the purpose of obtaining a geometrical description of this variety to the best known class of algebras of this type: the Poisson algebras.

The first historical encounter with Poisson algebras is in Hamiltonian mechanics (1809), where the algebra of smooth functions on variables the space coordinates $q_i$ and the corresponding momenta $p_i$ is endowed with a Lie bracket which satisfies a Leibniz compatibility rule, namely, the Lie bracket is a derivation. This is achieved by defining the product of two smooth functions by:
$$\left\{f, g\right\} = \sum_{i=1}^n (\frac{\partial f}{\partial p_{i}}\frac{\partial g}{\partial q_i} - \frac{\partial f}{\partial q_i}\frac{\partial g}{\partial p_i}).$$

This structure, the Hamiltonian and the Hamilton equations make a system that is widely used to describe any kind of environment in classical mechanics and that provides a solid theoretical foundation for the later generalization to quantum mechanics. Additionally, there are generalizations for $n$-ary Lie brackets which can handle the sue of $(n-1)$ Hamiltonians through the Nambu-Poisson algebras. These generalizations are motivated by different applications including string theory, M-branes theory and quarks models \cite{nambu, fili, takh}. 
With the later abstraction of many notions of classical Physics, 
Hamiltonian systems were geometrized into manifolds that model the set of all possible configurations of the system, and the cotangent bundle of this manifold describes its phase space, which is endowed with a Poisson algebra \cite{Weinstein}. This construction has been used in a wide range of fields in Physics, including classical mechanics, quantum mechanics, general relativity and geometrical optics, quantization theory, quantum groups, and in Mathematics, including representation theory, linear partial differential equations or completely integrable systems.

The study of all possible Poisson algebras with a certain Lie or associative part is an important problem in the theory of Poisson algebras. In the complex case, some examples of Poisson algebras were constructed on finite dimensional commutative associative algebras \cite{goze10} and a classification of the $3$-dimensional Poisson algebras was given by considering the Poisson structures defined on the $3$-dimensional Lie algebras \cite{gr06}. In this paper, we have considered the problem of classifying the Poisson structures of some finite-dimensional algebras. In particular, we correct a minor mistake in \cite{gr06} on the classification of $3$-dimensional Poisson algebras and we classify other interesting types of Poisson algebras. In addition, we consider the problem of the geometric classification of the varieties defined by this kind of structures in which we find two multiplications linked by a compatibility rule, closing this work with the geometric classification, the complete graph of degenerations and the description of the closures of the orbits of the variety of $3$-dimensional Poisson algebras.

Thus, this work is organized as follows. In the first section (\ref{sec1}), we introduce the classification method that we use to obtain the algebraic classification of the Poisson algebras using the classification of the commutative associative algebras. In the second section (\ref{sec2}), we apply this method to obtain the classification of the $3$-dimensional Poisson algebras. In the third section (\ref{sec3}), we classify Poisson algebras defined on commutative associative null-filiforms and filiforms. In the fourth section (\ref{sec4}), we study the geometry of the variety of Poisson algebras of dimension three. Finally, in the fifth section (\ref{sec5}), we study the degenerations between the Poisson algebras defined on commutative associative null-filiforms and filiforms.

\section{{\protect\Large   } The algebraic classification method}
\label{sec1}

In this section, we stablish a straightforward method to obtain the algebraic classification of the Poisson structures defined over an arbitrary commutative associative algebra. First, let us recall the definition of a Poisson algebra.

\begin{definition} 

A \emph{Poisson algebra}  is a  vector space $\mathcal{P}$ over an arbitrary field $\mathbb F$ equipped with two bilinear operations:

\begin{enumerate}
    \item An commutative associative multiplication, denoted by $-\cdot- :\mathcal{P}\times   \mathcal{P}\longrightarrow \mathcal{P}$;

    \item A Lie algebra multiplication, denoted by $\left\{-, -\right\} :\mathcal{P}\times   \mathcal{P}\longrightarrow \mathcal{P}$.
\end{enumerate}
These two operations are required to satisfy  the following Leibniz identity:
\begin{equation*}
\left\{x\cdot y, z\right\}=\left\{x,z\right\}\cdot y + x\cdot\left\{y,z\right\},
\end{equation*}
for any $x,y,z\in \mathcal{P}$.
  The \emph{dimension}  of the Poisson algebra $\mathcal{P}$ is the dimension of $\mathcal{P}$ as a vector space.
\end{definition}

In this paper, all vector spaces are
assumed to be complex and finite dimensional, unless stated otherwise. For simplicity, every time we write the multiplication table of a Poisson algebra the products of basic elements whose values are zero or can be recovered by the commutativity, in the case of $-\cdot-$, or by the anticommutativity, in the case of $\left\{-, -\right\}$, are omitted. Further, the variety of all $n$-dimensional Poisson algebras over the complex field will be denoted $\P_n$.

\medskip

Now, given an arbitrary commutative associative algebra, we may consider all the Poisson structures defined over this algebra. This notion is captured in the following definition.

\begin{definition}
Let $\left( \mathcal{P},\cdot \right) $ be a commutative associative
algebra. Define $Z^{2}\left( \mathcal{P},\mathcal{P}\right) $ to be the
set of all skew symmetric bilinear maps $\theta :\mathcal{P}\times 
\mathcal{P}\longrightarrow \mathcal{P}$ such that:%
\begin{eqnarray*}
\theta \left( \theta \left( x,y\right) ,z\right) +\theta \left( \theta
\left( y,z\right) ,x\right) +\theta \left( \theta \left( z,x\right)
,y\right)  &=&0, \\
\theta \left( x\cdot y,z\right) -\theta \left( x,z\right) \cdot y-x\cdot
\theta \left( y,z\right)  &=&0,
\end{eqnarray*}
for all $x,y,z$ in $\mathcal{P}$.  Then $Z^{2}\left(\mathcal{P},\mathcal{P}\right)\neq \mathcal{\varnothing }$ since $\theta=0\in Z^{2}\left(\mathcal{P},\mathcal{P}\right)$.
\end{definition}

Observe that, for $\theta \in Z^{2}\left( \mathcal{P},\mathcal{P}\right)$, if we define a bracket $\left\{
-,-\right\} _{\theta }$ on $\mathcal{P}$ by $\left\{ x,y\right\} _{\theta
}=\theta \left( x,y\right) $ for all $x,y$ in $\mathcal{P}$, then $\left( 
\mathcal{P},\cdot ,\left\{ -,-\right\} _{\theta }\right) $ is a Poisson
algebra. Conversely, if $\left( \mathcal{P},\cdot ,\left\{ -,-\right\}
\right) $ is a Poisson algebra, then there exists $\theta \in Z^{2}\left( 
\mathcal{P},\mathcal{P}\right) $ such that $\left( 
\mathcal{P},\cdot ,\left\{ -,-\right\} _{\theta }\right) \cong\left( \mathcal{P},\cdot ,\left\{ -,-\right\}
\right)$. To see this, consider the skew symmetric bilinear map $\theta :\mathcal{P}\times \mathcal{P%
}\longrightarrow \mathcal{P}$ defined by $%
\theta \left( x,y\right) =\left\{ x,y\right\} $ for all $x,y$ in $\mathcal{P}
$. Then $\theta \in Z^{2}\left( 
\mathcal{P},\mathcal{P}\right) $ and $\left( 
\mathcal{P},\cdot ,\left\{ -,-\right\} _{\theta }\right) =\left( \mathcal{P},\cdot ,\left\{ -,-\right\}
\right)$.

\medskip

Now, let $\left( \mathcal{P},\cdot \right) $ be a commutative associative
algebra and $\textrm{Aut}\left( \mathcal{P}\right) $ be the automorphism group
of $\mathcal{P}$. Then $\textrm{Aut}\left( \mathcal{P}\right) $ acts on $Z^{2}\left( \mathcal{P},\mathcal{P}\right) $ by $$
\left(\theta *\phi\right)  \left( x,y\right) =\phi ^{-1}\left( \theta \left( \phi \left(
x\right) ,\phi \left( y\right) \right) \right),$$ for $\phi \in \textrm{Aut}%
\left( \mathcal{P}\right) $, and $\theta \in Z^{2}\left( \mathcal{P},%
\mathcal{P}\right) $.
\begin{lemma}
Let $\left( \mathcal{P},\cdot \right) $ be a commutative associative algebra
and $\theta ,\vartheta \in Z^{2}\left( \mathcal{P},\mathcal{P}\right) $.
Then $\left( \mathcal{P},\cdot ,\left\{ -,-\right\} _{\theta }\right) $ and $%
\left( \mathcal{P},\cdot ,\left\{ -,-\right\} _{\vartheta }\right) $ are
isomorphic if and only if there is a $\phi \in \textrm{Aut}\left( \mathcal{P}%
\right) $ with $\theta *\phi =\vartheta $.
\end{lemma}

\begin{proof}
If $\theta *\phi =\vartheta  $, then $\phi :\left( \mathcal{P},\cdot ,\left\{
-,-\right\} _{\vartheta }\right) \longrightarrow $ $\left( \mathcal{P},\cdot
,\left\{ -,-\right\} _{\theta }\right) $ is an isomorphism since $\phi
\left( \vartheta \left( x,y\right) \right) =\theta \left( \phi \left(
x\right) ,\phi \left( y\right) \right) $. On the other hand, if $\phi
:\left( \mathcal{P},\cdot ,\left\{ -,-\right\} _{\vartheta }\right)
\longrightarrow $ $\left( \mathcal{P},\cdot ,\left\{ -,-\right\} _{\theta
}\right) $ is an isomorphism of Poisson algebras, then $\phi \in \textrm{Aut}%
\left( \mathcal{P}\right) $ and $\phi \left( \left\{ x,y\right\} _{\vartheta
}\right) =\left\{ \phi \left( x\right) ,\phi \left( y\right) \right\}
_{\theta }$. Hence $\vartheta \left( x,y\right) =\phi ^{-1}\left( \theta
\left( \phi \left( x\right) ,\phi \left( y\right) \right) \right)= \left(\theta *
\phi \right)\left( x,y\right) $ and therefore $\theta *\phi =\vartheta  $.
\end{proof}

\bigskip

Hence, we have a procedure to classify the Poisson algebras associated to a given commutative associative algebra $\left( \mathcal{P},\cdot \right) 
$. It consists of three steps:

\begin{enumerate}
\item Compute $Z^{2}\left( \mathcal{P},\mathcal{P}\right) $.

\item Find the orbits of $\textrm{Aut}\left( \mathcal{P}\right) $ on $%
Z^{2}\left( \mathcal{P},\mathcal{P}\right) $.
\item  Choose a representative $\theta$ from each orbit and then construct the Poisson algebra $\left( \mathcal{P},\cdot
,\left\{ -,-\right\} _{\theta }\right) $.
\end{enumerate}

\begin{remark}

Similarly, we can construct an analogous method for classifying the 3-dimensional Poisson algebras from the classification of Lie algebras of dimension three. However, since we are also looking for a classification of Poisson algebra with its associated commutative associative algebra to be either null-filiform or filiform, we have to consider the method that starts with commutative associative algebra.

\end{remark}

\bigskip

Let us introduce the following notations. Let $e_{1},e_{2},\ldots ,e_{n}$ be a fixed basis of a commutative
associative algebra $\left( \mathcal{P},\cdot \right) $. Define $\mathrm{%
\Lambda }^{2}\left( \mathcal{P},\mathbb{F}\right) $ to be the space of all
skew symmetric bilinear forms on $\mathcal{P}$. Then  $\mathrm{\Lambda }%
^{2}\left( \mathcal{P},\mathbb{F}\right) =\left\langle \Delta
_{i,j}:1\leq i<j\leq n\right\rangle $ where $\Delta _{i,j}$ is the skew-symmetric bilinear form $\Delta _{i,j}:\mathcal{P}\times \mathcal{P}%
\longrightarrow \mathbb{F}$\ defined by%
\[
\Delta _{i,j}\left( e_{l},e_{m}\right) :=\left\{ 
\begin{tabular}{ll}
$1,$ & if $\left( i,j\right) =\left( l,m\right) ,$ \\ 
$-1,$ & if $\left( i,j\right) =\left( m,l\right), $ \\ 
$0,$ & otherwise.%
\end{tabular}%
\right. 
\]
Now, if $\theta \in Z^{2}\left( \mathcal{P},\mathcal{%
P}\right) $, then $\theta $ can be uniquely written as $\theta \left(
x,y\right) =\underset{i=1}{\overset{n}{\sum }}B_{i}\left( x,y\right) e_{i}$
where $B_{1},B_{2},\ldots ,B_{n}$ is a sequence of skew symmetric bilinear
forms on $\mathcal{P}$. Also, we may write $\theta =\left(
B_{1},B_{2},\ldots ,B_{n}\right) $ . Let $\phi ^{-1}
\in \textrm{Aut}\left( \mathcal{P}\right) $ be given by the matrix $\left( b_{ij}\right)$. If $\left(\theta *\phi\right)  \left( x,y\right) =\underset{i=1}{\overset{n}{\sum }}B_{i}^{\prime }\left(
x,y\right) e_{i}$, then $B_{i}^{\prime }=\underset{j=1}{\overset{n}{\sum }}%
b_{ij}\phi ^{t}B_{j}\phi $.

\section{The algebraic classification of the 3-dimensional Poisson algebras}
\label{sec2}

To obtain the classification of the $3$-dimensional Poisson algebras, we will use the classification commutative associative algebras of dimension 3. This classification can be obtained from the  classification of 3-dimensional Jordan algebras given in \cite{ha17} . We have summarized it in the following result.

\begin{theorem}
\label{CAA}Let $\left( \mathcal{P},\left\{ -,-\right\} \right) $ be a complex
commutative associative algebra of dimension three. Then $\mathcal{P}$ has a basis $\left\{e_i\right\}_1^3$ and it is
isomorphic to one of the following algebras:

\begin{tasks}[style=itemize](2)
\task $A_{1}:$ trivial algebra.

\task $A_{2}:e_{1}\cdot e_{1}=e_{2}.$

\task $A_{3}:e_{1}\cdot e_{2}=e_{3}.$

\task $A_{4}:e_{1}\cdot e_{1}=e_{2},e_{1}\cdot e_{2}=e_{3}.$

\task $A_{5}:e_{1}\cdot e_{1}=e_{1},e_{2}\cdot e_{2}=e_{2},e_{3}\cdot
e_{3}=e_{3}.$

\task $A_{6}:e_{1}\cdot e_{1}=e_{1},e_{2}\cdot e_{2}=e_{2},e_{2}\cdot
e_{3}=e_{3}.$

\task $A_{7}:e_{1}\cdot e_{1}=e_{1},e_{1}\cdot e_{2}=e_{2},e_{1}\cdot
e_{3}=e_{3}.$

\task $A_{8}:e_{1}\cdot e_{1}=e_{1},e_{1}\cdot e_{2}=e_{2},e_{1}\cdot
e_{3}=e_{3},e_{2}\cdot e_{2}=e_{3}.$

\task $A_{9}:e_{1}\cdot e_{1}=e_{1},e_{2}\cdot e_{2}=e_{2}.$

\task $A_{10}:e_{1}\cdot e_{1}=e_{1},e_{1}\cdot e_{2}=e_{2}.$

\task $A_{11}:e_{1}\cdot e_{1}=e_{1}.$

\task $A_{12}:e_{1}\cdot e_{1}=e_{1},e_{2}\cdot e_{2}=e_{3}.$
\end{tasks}
\end{theorem}

Also, the following result (see \cite{ikv19})  will help us with the classification of 3-dimensional Poisson algebras when the commutative associative multiplication $(\mathcal{P}, \cdot)$ is trivial.

\begin{theorem}
\label{3-dim Lie}Let $\left( \mathcal{P},\left\{ -,-\right\} \right) $ be a complex
Lie algebra of dimension three. Then $\mathcal{P}$ is isomorphic to one
of the following algebras: 

\begin{itemize}
\item $\mathcal{L}_{3,1}:$ trivial algebra.

\item $\mathcal{L}_{3,2}:\left\{ e_{1},e_{2}\right\} =e_{3}.$

\item $\mathcal{L}_{3,3}:\left\{ e_{1},e_{2}\right\} =e_{2},\left\{
e_{1},e_{3}\right\} =e_{2}+e_{3}.$

\item $\mathcal{L}_{3,4}^{\alpha}:\left\{ e_{1},e_{2}\right\}
=e_{2},\left\{ e_{1},e_{3}\right\} =\alpha e_{3}.$

\item $\mathcal{L}_{3,5}:\left\{ e_{1},e_{2}\right\} =e_{3},\left\{
e_{1},e_{3}\right\} =-2e_{1},\left\{ e_{2},e_{3}\right\} =2e_{2}.$
\end{itemize}
\end{theorem}

Applying the method developed in the previous section of this work, we have obtained the algebraic classification of the three-dimensional complex Poisson algebras.

\begin{theorem}
\label{3-dim Poisson}Let $\left( \mathcal{P},\cdot ,\left\{ -,-\right\} \right) $ be a complex
Poisson algebra of dimension three. Then $\mathcal{P}$ is isomorphic to one
of the following Poisson algebras: 

\begin{tasks}[style=itemize](2)
\task $\mathcal{P}_{3,1}:$ trivial algebra.

\task $\mathcal{P}_{3,2}:\left\{ e_{1},e_{2}\right\} =e_{3}.$

\task $\mathcal{P}_{3,3}:\left\{ e_{1},e_{2}\right\} =e_{2},\left\{
e_{1},e_{3}\right\} =e_{2}+e_{3}.$

\task $\mathcal{P}_{3,4}^{\alpha}:\left\{ e_{1},e_{2}\right\}
=e_{2},\left\{ e_{1},e_{3}\right\} =\alpha e_{3}.$

\task $\mathcal{P}_{3,5}:\left\{ e_{1},e_{2}\right\} =e_{3},\left\{
e_{1},e_{3}\right\} =-2e_{1},\left\{ e_{2},e_{3}\right\} =2e_{2}.$

\task $\mathcal{P}_{3,6}:e_{1}\cdot e_{1}=e_{2},e_{1}\cdot e_{2}=e_{3}.$

\task $\mathcal{P}_{3,7}:e_{1}\cdot e_{1}=e_{1},e_{2}\cdot
e_{2}=e_{2},e_{3}\cdot e_{3}=e_{3}.$

\task $\mathcal{P}_{3,8}:e_{1}\cdot e_{1}=e_{1},e_{2}\cdot
e_{2}=e_{2},e_{2}\cdot e_{3}=e_{3}.$

\task $\mathcal{P}_{3,9}:e_{1}\cdot e_{1}=e_{1},e_{1}\cdot
e_{2}=e_{2},e_{1}\cdot e_{3}=e_{3},e_{2}\cdot e_{2}=e_{3}.$

\task $\mathcal{P}_{3,10}:e_{1}\cdot e_{1}=e_{1},e_{2}\cdot e_{2}=e_{2}.$

\task $\mathcal{P}_{3,11}:e_{1}\cdot e_{1}=e_{1},e_{1}\cdot e_{2}=e_{2}.$

\task $\mathcal{P}_{3,12}:e_{1}\cdot e_{1}=e_{1},e_{2}\cdot e_{2}=e_{3}.$

\task $\mathcal{P}_{3,13}:e_{1}\cdot e_{1}=e_{2}.$

\task $\mathcal{P}_{3,14}:\left\{ 
\begin{tabular}{l}
$e_{1}\cdot e_{1}=e_{2},$ \\ 
$\left\{ e_{1},e_{3}\right\} =e_{3}.$%
\end{tabular}%
\right. $

\task $\mathcal{P}_{3,15}:\left\{ 
\begin{tabular}{l}
$e_{1}\cdot e_{1}=e_{2},$ \\ 
$\left\{ e_{1},e_{3}\right\} =e_{2}.$%
\end{tabular}%
\right. $

\task $\mathcal{P}_{3,16}^{\alpha }:\left\{ 
\begin{tabular}{l}
$e_{1}\cdot e_{2}=e_{3},$ \\ 
$\left\{ e_{1},e_{2}\right\} =\alpha e_{3}.$%
\end{tabular}%
\right. $

\task $\mathcal{P}_{3,17}:e_{1}\cdot e_{1}=e_{1},e_{1}\cdot
e_{2}=e_{2},e_{1}\cdot e_{3}=e_{3}.$

\task $\mathcal{P}_{3,18}:\left\{ 
\begin{tabular}{l}
$e_{1}\cdot e_{1}=e_{1},e_{1}\cdot e_{2}=e_{2},e_{1}\cdot e_{3}=e_{3},$ \\ 
$\left\{ e_{2},e_{3}\right\} =e_{2}.$%
\end{tabular}%
\right. $

\task $\mathcal{P}_{3,19}:e_{1}\cdot e_{1}=e_{1}.$

\task $\mathcal{P}_{3,20}:\left\{ 
\begin{tabular}{l}
$e_{1}\cdot e_{1}=e_{1},$ \\ 
$\left\{ e_{2},e_{3}\right\} =e_{2}.$%
\end{tabular}%
\right. $
\end{tasks}

Between these algebras there are precisely the following isomorphisms: 
\begin{itemize}
    \item $\mathcal{P}_{3,4}^{\alpha\neq 0}\cong \mathcal{P}_{3,4}^{\beta\neq 0}$ if and only if $\alpha=\beta$ or $\alpha=\frac{1}{\beta}$.
    
    \item $\mathcal{P}_{3,16}^{\alpha}\cong \mathcal{P}_{3,16}^{\beta}$ if and only if $\alpha^{2}=\beta^{2}$.
\end{itemize}

\end{theorem}

\begin{proof}
By Theorem \ref{CAA}, we may assume $\left( \mathcal{P},\cdot \right) \in
\left\{ A_{1},A_{2},\ldots ,A_{12}\right\} $. If $\left( \mathcal{P},\cdot
\right) =A_{1}$, then $\left( \mathcal{P},\{-,-\} \right) $ is a Lie algebra and we get the algebras $\mathcal{P}_{3,1},\mathcal{P}%
_{3,2},\mathcal{P}_{3,3},\mathcal{P}_{3,4}^{\alpha}$
and $\mathcal{P}_{3,5}$. If $\left( \mathcal{P},\cdot \right) \in \left\{
A_{4},A_{5},A_{6},A_{8},A_{9},A_{10},A_{12}\right\} $, then $Z^{2}\left( 
\mathcal{P},\mathcal{P}\right) =\{0\}$. So we get the algebras $\mathcal{P}%
_{3,6},\ldots ,\mathcal{P}_{3,12}$. Assume now that $\left( \mathcal{P}%
,\cdot \right) \in \left\{ A_{2},A_{3},A_{7},A_{11}\right\} $. Then we have
the following cases:

\begin{enumerate}
\item $\left( \mathcal{P},\cdot \right) =A_{2}$. Let $\theta
=\left( B_{1},B_{2},B_{3}\right) $ be an arbitrary element of $Z^{2}\left( \mathcal{P},\mathcal{P}%
\right) $. Then $\theta=\left( 0,\alpha \Delta
_{1,3},\beta \Delta _{1,3}\right) $ for some $\alpha,\beta \in \mathbb{C}$. The automorphism group of $A_{2}$, an element $\phi \in
\textrm{Aut}\left( A_{2}\right) $ is given by an invertible matrix of the
following form:%
\[
\left( 
\begin{array}{ccc}
a_{11} & 0 & 0 \\ 
a_{21} & a_{11}^{2} & a_{23} \\ 
a_{31} & 0 & a_{33}%
\end{array}%
\right) .
\]%
Suppose that $\theta *\phi
 =\left( 0,\alpha ^{\prime }\Delta _{1,3},\beta ^{\prime
}\Delta _{1,3}\right) $. Then%
\begin{eqnarray*}
\alpha ^{\prime } &=&\frac{1}{a_{11}}\left( \alpha a_{33}-\beta
a_{23}\right) , \\
\beta ^{\prime } &=&\beta a_{11}.
\end{eqnarray*}
Let us distinguish two cases:
\begin{itemize}
\item $\left( \alpha ,\beta \right) =\left( 0,0\right) $. In this case we
get the algebra $\mathcal{P}_{3,13}$.

\item $\left( \alpha ,\beta \right) \neq \left( 0,0\right) $.  If $\beta
\neq 0$, we get the representative $\theta =\left( 0,0,\Delta
_{1,3}\right) $ \ and therefore we get the algebra $\mathcal{P}_{3,14}$. On
the other hand, if $\beta =0$, we obtain the representative $\theta =\left(
0,\Delta _{1,3},0\right) $\ and thus we get the algebra $\mathcal{P}%
_{3,15}$.
\end{itemize}

\item $\left( \mathcal{P},\cdot \right) =A_{3}$. If $\theta =\left(
B_{1},B_{2},B_{3}\right)$ is an arbitrary element of $Z^{2}\left( \mathcal{P},\mathcal{P}\right) $,
then $\theta=\left( 0,0,\alpha \Delta
_{1,2}\right) $ for some $\alpha \in \mathbb{C}$. Moreover, the automorphism group of $A_{2}$, $\textrm{Aut}\left(
A_{2}\right) $, consists of the isomorphisms $\phi$ given by a matrix of the following form:%
\[
\left( 
\begin{array}{ccc}
a_{11} & a_{12} & 0 \\ 
a_{21} & a_{22} & 0 \\ 
a_{31} & a_{32} & a_{11}a_{22}+a_{12}a_{21}%
\end{array}%
\right) :a_{12}=a_{21}=0\text{ or }a_{11}=a_{22}=0.
\]%
Assume that $\theta *\phi =\left( 0,0,\alpha ^{\prime }\Delta _{1,2}\right) $. Then%
\[
\alpha ^{\prime }=\frac{1}{a_{11}a_{22}+a_{12}a_{21}}\left( \alpha
a_{11}a_{22}-\alpha a_{12}a_{21}\right) .
\]%
From here we have $\alpha ^{\prime 2}=\alpha ^{2}$. So we get the
representatives $\theta ^{\alpha }=\left( 0,0,\alpha \Delta _{1,2}\right) 
$. Moreover, we have $\theta ^{\alpha }$and $\theta ^{\alpha ^{\prime }}$ in
the same orbit if and only if $\alpha ^{\prime 2}=\alpha ^{2}$. We denote
the algebras corresponding to the representatives $\theta ^{\alpha }$ by $%
\mathcal{P}_{3,16}^{\alpha }$.

\item $\left( \mathcal{P},\cdot\right) =A_{7}$. Choose an arbitrary element $\theta
=\left( B_{1},B_{2},B_{3}\right) \in Z^{2}\left( \mathcal{P},\mathcal{P}%
\right) $. Then $\theta=\left( 0,\alpha \Delta
_{2,3},\beta \Delta _{2,3}\right) $ for some $\alpha,\beta \in \mathbb{C}$. Furthermore, an automorphism $\phi \in
\textrm{Aut}\left( A_{7}\right) $, is given by an invertible matrix of
the form:%
\[
\left( 
\begin{array}{ccc}
1 & 0 & 0 \\ 
0 & a_{22} & a_{23} \\ 
0 & a_{32} & a_{33}%
\end{array}%
\right) .
\]%
Write $\theta *\phi =\left( 0,\alpha ^{\prime }\Delta _{2,3},\beta ^{\prime
}\Delta _{2,3}\right) $. Then
\begin{eqnarray*}
\alpha ^{\prime } &=&\alpha a_{33}-\beta a_{23}, \\
\beta ^{\prime } &=&\beta a_{22}-\alpha a_{32}.
\end{eqnarray*}%
If $\left( \alpha ,\beta \right) =\left( 0,0\right) $, we get the algebra $%
\mathcal{P}_{3,17}$. If $\left( \alpha ,\beta \right) \neq \left( 0,0\right) 
$, we get the representative $\theta =\left( 0,\Delta _{2,3},0\right) $
and hence we get the algebra $\mathcal{P}_{3,18}$.

\item $\left( \mathcal{P},\cdot \right) =A_{11}$. Consider an arbitrary element $\theta
=\left( B_{1},B_{2},B_{3}\right) \in Z^{2}\left( \mathcal{P},\mathcal{P}%
\right) $. Then $\theta=\left( 0,\alpha \Delta
_{2,3},\beta \Delta _{2,3}\right) $  for some $\alpha,\beta \in \mathbb{C}$. An automorphism $\phi \in\textrm{Aut}\left( A_{11}\right) $, is given by an invertible matrix of the form:%
\[
\left( 
\begin{array}{ccc}
1 & 0 & 0 \\ 
0 & a_{22} & a_{23} \\ 
0 & a_{32} & a_{33}%
\end{array}%
\right) .
\]%
Suppose that $\theta *\phi  =\left( 0,\alpha ^{\prime }\Delta _{2,3},\beta ^{\prime
}\Delta _{2,3}\right) $. Then%
\begin{eqnarray*}
\alpha ^{\prime } &=&\alpha a_{33}-\beta a_{23}, \\
\beta ^{\prime } &=&\beta a_{22}-\alpha a_{32}.
\end{eqnarray*}%
If $\left( \alpha ,\beta \right) =\left( 0,0\right) $, we get the algebra $%
\mathcal{P}_{3,19}$. If $\left( \alpha ,\beta \right) \neq \left( 0,0\right) 
$, we get the representative $\theta =\left( 0,\Delta _{2,3},0\right) $
and hence we get the algebra $\mathcal{P}_{3,20}$.
\end{enumerate}
These cases complete the classification of Poisson algebras of dimension three.
\end{proof}

\bigskip

Additionally, we have compared the classification we have obtained to the previously known classification, from the paper by Goze and Remm \cite{gr06}, and we have found the following isomorphisms, maintaining the corresponding notation of each work.

\begin{table}[H]
\centering
\begin{tabular}{|c|c|c|c|c|c|c|c|c|c|}
\hline
GR & $\mathcal{P}_{3,1}\left( \gamma \right) $ & $\mathcal{P}_{3,2}$ & $%
\mathcal{P}_{3,3}\left( \alpha \right) $ & $\mathcal{P}_{3,4}$ & $\mathcal{P}%
_{3,5}$ & $\mathcal{P}_{3,6}$ & $\mathcal{P}_{3,7}\left( \alpha \right) $ & $%
\mathcal{P}_{3,8}$ & $\mathcal{P}_{3,9}$ \\ \hline
AFM & $\mathcal{P}_{3,6}^{\alpha }$ & $\mathcal{P}_{3,15}$ & $\mathcal{P}%
_{3,20}$ & $\mathcal{P}_{3,14}$ & $\mathcal{P}_{3,2}$ & $\mathcal{P}_{3,18}$
& $\mathcal{P}_{3,4}^{\alpha}$ & $\mathcal{P}_{3,3}$ & $\mathcal{P}%
_{3,5}$ \\ \hline
\end{tabular}    
\caption{Comparison between Goze's and Remms classification and our classification.}
\end{table}

From the table, there is a family on that paper which is isomorphic to a single algebra from our classification.
\medskip

Finally, we recall the notion of Malcev Poisson algebra, which is a generalization of the notion of Poisson algebra.

\begin{definition}
A \emph{Malcev Poisson algebra}  is a  vector space $\mathcal{P}$ equipped with two bilinear operations:

\begin{enumerate}
    \item An commutative associative multiplication, denoted by $-\cdot- :\mathcal{P}\times   \mathcal{P}\longrightarrow \mathcal{P}$;

    \item A Malcev algebra multiplication, denoted by $\left\{-, -\right\} :\mathcal{P}\times   \mathcal{P}\longrightarrow \mathcal{P}$.
\end{enumerate}
These two operations are required to satisfy  the following Leibniz identity:
\begin{equation*}
\left\{x\cdot y, z\right\}=\left\{x,z\right\}\cdot y + x\cdot\left\{y,z\right\},
\end{equation*}
for any $x,y,z\in \mathcal{P}$.
\end{definition}

Since any Malcev algebra up to dimension three is a Lie algebra, we have the following result.

\begin{theorem}
All Malcev Poisson algebras of dimension three are Poisson algebras. Therefore, the classification of complex Malcev Poisson algebras of dimension three is given in Theorem \ref{3-dim Poisson}.
\end{theorem}

\section{The classification of certain types of Poisson algebras}
\label{sec3}

In this section, we consider the classification problem for an arbitrary dimension in two particular cases. Namely, we study the classification of all $n$-dimensional ($n>3$) complex Poisson algebra $\left( \mathcal{P}%
,\cdot ,\left\{ -,-\right\} \right) $ such that $\left( \mathcal{P},\cdot
\right) $ is a null-filiform algebra or a filiform algebra. Recall the definition of a null-filiform algebra.

\begin{definition}

An $n$-dimensional algebra $\mathcal{P}$ is a null-filiform if $\mathrm{dim}\,\mathcal{P}^i = n-i+1 $ where $2\leq i \leq n+1$ and 
$$\mathcal{P}^{1}= \mathcal{P}, \mathcal{P}^{k+1}= \mathcal{P}^{k} \cdot \mathcal{P} \textrm{ for } n\in {\mathbb N}.$$

\end{definition}

By \cite{D99, K21, M14}, we have the classification of the associative null-filiform  algebras.

\begin{theorem}
\label{nul-fili}An arbitrary $n$-dimensional null-filiform associative
algebra is isomorphic to the algebra: 
\begin{equation*}
\mu _{0}^{n}:\quad e_{i}\cdot e_{j}=e_{i+j},\quad 2\leq i+j\leq n,
\end{equation*}%
where $\{e_{1},e_{2},\dots ,e_{n}\}$ is a basis of the algebra $\mu _{0}^{n}$%
.
\end{theorem}

\begin{lemma}
\label{Z^2 for nul-filiform} Let $\left( \mathcal{P},\cdot \right) $\ be an $%
n $-dimensional null-filiform commutative associative algebra. Then $%
Z^{2}\left( \mathcal{P},\mathcal{P}\right) =\left\{ 0\right\} $.
\end{lemma}
\begin{proof}
By Theorem \ref{nul-fili}, we may assume $\left( \mathcal{P},\cdot \right)
=\mu _{0}^{n}$. Let $\theta $ be an arbitrary element of $Z^{2}\left( 
\mathcal{P},\mathcal{P}\right) $. Then, we will prove that $\theta \left( e_{i},e_{j}\right) =0$ for $%
1\leq i,j\leq n$\ by induction on $i+j$. If $i+j=2$ we have $\theta \left(
e_{i},e_{j}\right) =0$. Now, assume that $\theta \left( e_{i},e_{j}\right) =0$
for $i+j<n$. Then, for $i+j=n$, we have $i>1$ or $j>1$. Without loss of
generality, we may assume $i>1$. Hence, since $i-1+j<n$ and $1+j<n$, $\theta \left( e_{i},e_{j}\right)
=\theta \left( e_{i-1}\cdot e_{1},e_{j}\right) =\theta \left(
e_{i-1},e_{j}\right) \cdot e_{1}+e_{i-1}\cdot \theta \left(
e_{1},e_{j}\right) =0$.
\end{proof}
\begin{theorem}
\label{Poisson with null-filiform}Let $\left( \mathcal{P},\cdot ,\left\{
-,-\right\} \right) $\ be an $n$-dimensional Poisson algebra such that $%
\left( \mathcal{P},\cdot \right) $ is null-filiform. Then $\left( 
\mathcal{P},\cdot ,\left\{ -,-\right\} \right) $ is isomorphic to the
following algebra:%
\begin{equation*}
\mathcal{P}_{0}^{n}:\quad e_{i}\cdot e_{j}=e_{i+j},\quad 2\leq i+j\leq n,
\end{equation*}%
where $\{e_{1},e_{2},\dots ,e_{n}\}$ is a basis of the algebra $\mathcal{P}%
_{0}^{n}$.
\end{theorem}

\bigskip

Now, recall the definition of a filiform algebra.

\begin{definition}

An $n$-dimensional algebra $\mathcal{P}$ is a filiform if $\mathrm{dim}\,\mathcal{P}^i = n-i $ where $2\leq i \leq n$.

\end{definition}

Thanks to \cite{K20, K21}, we have the classification of the commutative associative filiform  algebras .

\begin{theorem}
\label{filiform}Every $n$-dimensional ($n>3$) complex filiform commutative
associative algebra is isomorphic to one of the next non-isomorphic
algebras with basis $\{e_{1},e_{2},\dots ,e_{n}\}$: 
\begin{equation*}
\begin{array}{llll}
\mu _{1,1}^{n}=\mu _{0}^{n-1}\oplus \mathbb{C}e_{n} & : & e_{i}\cdot
e_{j}=e_{i+j}, &  \\ 
\mu _{1,2}^{n} & : & e_{i}\cdot e_{j}=e_{i+j}, & e_{n}\cdot e_{n}=e_{n-1}.%
\end{array}%
\end{equation*}%
where $2\leq i+j\leq n-1$.
\end{theorem}

\begin{lemma}
\label{Z2 for filiform}Let $\left( \mathcal{P},\cdot \right) $ be an $n$%
-dimensional ($n>3$) complex filiform commutative associative algebra.

\begin{enumerate}
\item If $\left( \mathcal{P},\cdot \right) =\mu _{1,1}^{n}$, then $%
Z^{2}\left( \mathcal{P},\mathcal{P}\right) =\left\{ \alpha \Delta
_{1,n}\left( -,-\right) e_{n-1}+\beta \Delta _{1,n}\left( -,-\right)
e_{n}:\alpha ,\beta \in 
\mathbb{C}
\right\} $.

\item If $\left( \mathcal{P},\cdot \right) =\mu _{1,2}^{n}$, then $%
Z^{2}\left( \mathcal{P},\mathcal{P}\right) =\left\{ \alpha \Delta
_{1,n}\left( -,-\right) e_{n-1}:\alpha \in 
\mathbb{C}
\right\} $.
\end{enumerate}
\end{lemma}

\begin{proof}
$\left( 1\right) $ Let $\theta \in Z^{2}\left( \mathcal{P},\mathcal{P}%
\right) $. As in Lemma \ref{nul-fili}, we have $\theta \left(
e_{i},e_{j}\right) =0$ for $1\leq i,j\leq n-1$. Since $\theta \left( x\cdot
y,z\right) =x\cdot \theta \left( y,z\right) +y\cdot \theta \left( x,z\right) 
$, we have $x\cdot \theta \left( e_{n},z\right) =0$ for all $x,z\in \mu
_{1,1}^{n}$. From here, we conclude that $\theta \left( e_{i},e_{n}\right)
\in \left\langle e_{n-1},e_{n}\right\rangle $ for $1\leq i\leq n$. Further
if $i>1$, then we have%
\begin{equation*}
\theta \left( e_{i},e_{n}\right) =e_{1}\cdot \theta \left(
e_{i-1},e_{n}\right) +e_{i-1}\cdot \theta \left( e_{1},e_{n}\right) =0.
\end{equation*}%
Hence $\theta \left( x,y\right) =\alpha \Delta _{1,n}\left( x,y\right)
e_{n-1}+\beta \Delta _{1,n}\left( x,y\right) e_{n}$ for some $\alpha ,\beta
\in 
\mathbb{C}
$.

$\left( 2\right) $ The proof is similar to the above case.
\end{proof}

\medskip

The group of automorphisms of the algebras $\mu _{1,1}^{n}$ and $\mu _{1,2}^{n}$ is given by \cite{K21}.

\begin{lemma}
Let $\phi _{1,s}^{n}\in \textrm{Aut}(\mu _{1,s}^{n})$. Then $\phi _{1,1}^{n}, \phi _{1,2}^{n}$ and are given respectively by the matrices 
\begin{equation*}
\begin{pmatrix}
a_{1,1} &  &  &  &  & 0 \\ 
& a_{1,1}^{2} &  &  &  & 0 \\ 
&  & a_{1,1}^{3} &  &  & \vdots \\ 
& \ast &  &  &  & 0 \\ 
&  &  &  & a_{1,1}^{n-1} & a_{n-1,n} \\ 
a_{n,1} & 0 & \dots & 0 & 0 & a_{n,n} \\ 
&  &  &  &  & 
\end{pmatrix}%
,
\begin{pmatrix}
a_{1,1} &  &  &  &  & 0 \\ 
& a_{1,1}^{2} &  &  &  & \vdots \\ 
&  & a_{1,1}^{3} &  &  & 0 \\ 
& \ast &  &  &  & -a_{n,1}a_{1,1}^{(n-3)/2} \\ 
&  &  &  & a_{1,1}^{n-1} & a_{n-1,n} \\ 
a_{n,1} & 0 & \dots & 0 & 0 & a_{1,1}^{(n-1)/2} \\ 
&  &  &  &  & 
\end{pmatrix}%
.
\end{equation*}
\end{lemma}

Finally, we address the isomorphism problem.

\begin{theorem}
\label{Poisson with filiform}Let $\left( \mathcal{P},\cdot ,\left\{
-,-\right\} \right) $\ be an $n$-dimensional ($n>3$) complex Poisson algebra such that $%
\left( \mathcal{P},\cdot \right) $\ is filiform. Then $\left( 
\mathcal{P},\cdot ,\left\{ -,-\right\} \right) $ is isomorphic to one of the
following algebras:

\begin{itemize}
\item $\mathcal{P}_{1,1}^{n}=P_{0}^{n-1}\oplus \mathbb{C}e_{n}:e_{i}\cdot
e_{j}=e_{i+j}.$

\item $\mathcal{P}_{1,2}^{n}:\left\{ 
\begin{array}{c}
e_{i}\cdot e_{j}=e_{i+j}, \\ 
\left\{ e_{1},e_{n}\right\} =e_{n}.%
\end{array}%
\right. $

\item $\mathcal{P}_{1,3}^{n}:\left\{ 
\begin{array}{c}
e_{i}\cdot e_{j}=e_{i+j}, \\ 
\left\{ e_{1},e_{n}\right\} =e_{n-1}.%
\end{array}%
\right. $

\item $\mathcal{P}_{1,4}^{n}:e_{i}\cdot e_{j}=e_{i+j},e_{n}\cdot
e_{n}=e_{n-1}.$

\item $\mathcal{P}_{1,5}^{n}:\left\{ 
\begin{array}{c}
e_{i}\cdot e_{j}=e_{i+j},e_{n}\cdot e_{n}=e_{n-1}, \\ 
\left\{ e_{1},e_{n}\right\} =e_{n-1}.%
\end{array}%
\right. $
\end{itemize}
where $2\leq i+j\leq n-1$.
\end{theorem}
\begin{proof}
By Theorem \ref{filiform}, we may assume $\left( \mathcal{P},\cdot \right)
\in \left\{ \mu _{1,1}^{n},\mu _{1,2}^{n}\right\} $. Therefore, we have the
following cases:

\begin{enumerate}
\item $\left( \mathcal{P},\cdot \right) =\mu _{1,1}^{n}$. Choose an
arbitrary element $\theta $\ of $Z^{2}\left( \mathcal{P},\mathcal{P}\right) $%
. Then $\theta \left( x,y\right) =\alpha \Delta _{1,n}\left( x,y\right)
e_{n-1}+\beta \Delta _{1,n}\left( x,y\right) e_{n}$. Consider $\phi \in
 \textrm{Aut}(\mu _{1,1}^{n})$ given by $\big(a_{i,j}\big)$, and let $\theta *\phi \left( x,y\right) =\alpha
^{\prime }\Delta _{1,n}\left( x,y\right) e_{n-1}+\beta ^{\prime }\Delta
_{1,n}\left( x,y\right) e_{n}$. Since $\theta *\phi \left( e_{1},e_{n}\right)
=\phi ^{-1}\left( \theta \left( \phi(e_{1}),\phi(e_{n})\right) \right) $, we have%
\begin{equation*}
\alpha ^{\prime }e_{n-1}+\beta ^{\prime }e_{n}=\alpha a_{1,1}a_{n,n}\phi
^{-1}\left( e_{n-1}\right) +\beta a_{1,1}a_{n,n}\phi ^{-1}\left(
e_{n}\right) .
\end{equation*}%
Moreover, we have $\phi ^{-1}\left( e_{n-1}\right) =a_{1,1}^{1-n}%
e_{n-1}$ and $\phi ^{-1}\left( e_{n}\right) =-a_{n-1,n}a_{1,1}^{1-n}%
e_{n-1}+\frac{1}{a_{n,n}}e_{n}$. Thus%
\begin{eqnarray*}
\alpha ^{\prime } &=&a_{1,1}^{2-n}\left( \alpha
a_{n,n}- \beta a_{n-1,n}\right) , \\
\beta ^{\prime } &=&\beta a_{1,1}.
\end{eqnarray*}%
If $\left( \alpha ,\beta \right) =0$, we get the algebra $\mathcal{P}%
_{1,1}^{n}$. Assume now $\left( \alpha ,\beta \right) \neq 0$. If $\beta
\neq 0$, we get the representative $\theta \left( x,y\right) =\Delta
_{1,n}\left( x,y\right) e_{n}$. So we obtain the algebra $\mathcal{P}%
_{1,2}^{n}$. If $\beta =0$, we get the representative $\theta \left(
x,y\right) =\Delta _{1,n}\left( x,y\right) e_{n-1}$. Thus we have the
algebra $\mathcal{P}_{1,3}^{n}$.

\item $\left( \mathcal{P},\cdot \right) =\mu _{1,2}^{n}$. Choose an
arbitrary element $\theta $\ of $Z^{2}\left( \mathcal{P},\mathcal{P}\right) $%
. Then $\theta \left( x,y\right) =\alpha \Delta _{1,n}\left( x,y\right)
e_{n-1}$. Consider $\phi 
\in \textrm{Aut}(\mu _{1,2}^{n})$ given by $\big(a_{i,j}\big)$, and let $\theta *\phi =\alpha ^{\prime }\Delta
_{1,n}e_{n-1}$. Since $\theta *\phi \left( e_{1},e_{n}\right) =\phi
^{-1}\left( \theta \left( \phi(e_{1}), \phi(e_{n})\right) \right) $, we have%
\begin{equation*}
\alpha ^{\prime }e_{n-1}=\alpha a_{1,1}a_{1,1}^{(n-1)/2}\phi ^{-1}\left(
e_{n-1}\right) =\alpha a_{1,1}^{(3-n)/2}e_{n-1}.
\end{equation*}%
So we have $\alpha ^{\prime }=\alpha a_{1,1}^{(3-n)/2}$. If 
$\alpha =0$, we get the algebra $\mathcal{P}_{1,4}^{n}$. If $\alpha \neq 0$,
we obtain the representative $\theta \left( x,y\right) =\Delta _{1,n}\left(
x,y\right) e_{n-1}$. Therefore we have the algebra $\mathcal{P}_{1,5}^{n}$.
\end{enumerate} 
In short, we obtain five algebras, up to isomorphisms, under the conditions of the theorem.
\end{proof}

\begin{remark}
As the proof of Lemma \ref{Z^2 for nul-filiform} and Lemma \ref{Z2 for
filiform} only depends  on the identity $\theta \left( x\cdot y,z\right)
=x\cdot \theta \left( y,z\right) +y\cdot \theta \left( x,z\right) $, Theorem %
\ref{Poisson with null-filiform} and Theorem \ref{Poisson with filiform} are also
true for Malcev Poisson algebras.
\end{remark}

\section{The geometric classification of the 3-dimensional Poisson algebras}
\label{sec4}

Given a complex vector space ${\mathbb V}$ of dimension $n$, the set of bilinear maps $\textrm{Bil}({\mathbb V} , {\mathbb V}) \cong \textrm{Hom}({\mathbb V} ^{\otimes2}, {\mathbb V})\cong ({\mathbb V}^*)^{\otimes2} \otimes {\mathbb V}$ is a vector space of dimension $n^3$. The set of pairs of bilinear maps (or bilinear pairs) $\textrm{Bil}({\mathbb V} , {\mathbb V}) \oplus \textrm{Bil}({\mathbb V} , {\mathbb V}) \cong ({\mathbb V}^*)^{\otimes2} \otimes {\mathbb V} \oplus ({\mathbb V}^*)^{\otimes2} \otimes {\mathbb V}$ which is a vector space of dimension $2n^3$. This vector space has the structure of the affine space $\mathbb{C}^{2n^3}$ in the following sense:
fixed a basis $e_1, \ldots, e_n$ of ${\mathbb V}$, then any pair with multiplication $(\mu, \mu')$, is determined by some parameters $c_{ij}^k, c_{ij}'^k \in \mathbb{C}$,  called {structural constants},  such that
$$\mu(e_i, e_j) = \sum_{p=1}^n c_{ij}^k e_k \textrm{ and } \mu'(e_i, e_j) = \sum_{p=1}^n c_{ij}'^k e_k$$
which corresponds to a point in the affine space $\mathbb{C}^{2n^3}$. Then a set of bilinear pairs $\mathcal S$ corresponds to an algebraic variety, i.e., a Zariski closed set, if there are some polynomial equations in variables $c_{ij}^k, c_{ij}'^k$ with zero locus equal to the set of structural constants of the bilinear pairs in $\mathcal S$. Since given the identities defining  Poisson algebras we can obtain a set of polynomial equations in variables $c_{ij}^k, c_{ij}'^k$, the class of $n$-dimensional  Poisson algebras $\mathcal{P}_{n}$ is a variety. 
Now, consider the following action of $\textrm{GL}({\mathbb V})$ on $\P_{n}$:
$$(g*(\mu, \mu'))(x,y) := (g \mu (g^{-1} x, g^{-1} y), g \mu' (g^{-1} x, g^{-1} y))$$
for $g\in\textrm{GL}({\mathbb V})$, $(\mu, \mu')\in \P_{n}$ and for any $x, y \in {\mathbb V}$. Observe that the $\textrm{GL}({\mathbb V})$-orbit of $(\mu, \mu')$, denoted $O((\mu, \mu'))$, contains all the structural constants of the bilinear pairs isomorphic to the  Poisson algebra with structural constants $(\mu, \mu')$.

In the previous section, we gave a decomposition of ${\P_3}$ into $\textrm{GL}(\mathbb V)$-orbits, i.e., an algebraic classification of the 3-dimensional Poisson algebras. In this section, we will describe the closures of orbits of $(\mu, \mu')\in\P_{3}$, denoted by $\overline{O((\mu, \mu'))}$, and we will give a geometric classification of ${\P_3}$, which consists in describing its irreducible components. Recall that any affine variety can be represented as a finite union of its irreducible components in a unique way.

Additionally, describing the irreducible components of a variety, such as ${\P_3}$, gives us 
which are those bilinear pairs with an open $\textrm{GL}(\mathbb V)$-orbit. 

\begin{definition}
\rm Let $\mathcal{P} $ and $\mathcal{P}'$ be two $n$-dimensional  Poisson algebras and $(\mu, \mu'), (\lambda,\lambda') \in \P_{n}$ be their representatives in the affine space, respectively. We say $\mathcal{P}$ {degenerates}  to $\mathcal{P}'$, and write $\mathcal{P} \to \mathcal{P} '$, if $(\lambda,\lambda')\in\overline{O((\mu, \mu'))}$. If $\mathcal{P}  \not\cong \mathcal{P}'$, then we call it a  {proper degeneration}.

Conversely, if $(\lambda,\lambda')\not\in\overline{O((\mu, \mu'))}$ then we call it a 
{non-degeneration} and we write ${\mathcal{P} }\not\to {\mathcal{P} }'$.
\end{definition}

\noindent Note that the definition of a degeneration does not depend on the choice of $(\mu, \mu')$ and $(\lambda,\lambda')$. Also, due to the transitivity of the notion of degeneration  (that is, if ${\mathcal{P} }\to {\mathcal{P} }''$ and ${\mathcal{P} }''\to {\mathcal{P} }'$ then ${\mathcal{P} }\to{\mathcal{P} }'$) we have the following definitions.

\begin{definition} \rm
Let $\mathcal{P} $ and $\mathcal{P} '$ be two $n$-dimensional  Poisson algebras such that ${\mathcal{P} } \to {\mathcal{P} }'$. If there is no ${\mathcal{P} }''$ such that ${\mathcal{P} }\to {\mathcal{P} }''$ and ${\mathcal{P} }''\to {\mathcal{P} }'$ are proper degenerations, then ${\mathcal{P} }\to {\mathcal{P} }'$ is called a  primary degeneration. Analogously, let $\mathcal{P} $ and $\mathcal{P} '$ be two $n$-dimensional  Poisson algebras such that $\mathcal{P} \not\to \mathcal{P} '$, if there are no ${\mathcal{P} }''$ and ${\mathcal{P} }'''$ such that ${\mathcal{P} }''\to {\mathcal{P} }$, ${\mathcal{P} }'\to {\mathcal{P} }'''$, ${\mathcal{P} }''\not\to {\mathcal{P} }'''$ and one of the assertions ${\mathcal{P} }''\to {\mathcal{P} }$ and ${\mathcal{P} }'\to {\mathcal{P} }'''$ is a proper degeneration,  then ${\mathcal{P} } \not\to {\mathcal{P} }'$ is called a  { primary non-degeneration}.
\end{definition}

Note that it suffices to prove primary degenerations and non-degenerations to fully describe the geometry of a variety. Therefore, in this work we will focus on proving the primary degenerations and non-degenerations.

{
Firstly, if $\mathfrak{Der}( \mathcal{P} )$ denotes the Lie algebra of derivations of  $\mathcal{P} $, $\mathrm{dim}\, \mathfrak{Der}( \mathcal{P} )$ is equal to $\mathrm{dim}\, \textrm{Aut}( \mathcal{P} )$, as an algebraic group. Recall the formula $\mathrm{dim}\, Gx = \mathrm{dim}\, G - \mathrm{dim}\,\mathrm{Stab}(x)$, where $G$ is an algebraic group acting on a variety $X$, $x\in X$, $Gx$ denotes the orbit of $x$ and $\mathrm{Stab}(x)$ denotes the stabilizer of $x$. Then,  $\mathrm{dim}\,O((\mu, \mu')) = n^2 - \mathrm{dim}\,\mathfrak{Der}(\mathcal{P})$. Therefore, if $ \mathcal{P} \to  \mathcal{P} '$ and  $\mathcal{P} \not\cong  \mathcal{P} '$, we have that $\mathrm{dim}\,\mathfrak{Der}( \mathcal{P} )<\mathrm{dim}\,\mathfrak{Der}( \mathcal{P} ')$. Hence, we will check the assertion ${\mathcal{P} }\to {\mathcal{P} }'$ only for ${\mathcal{P} }$ and ${\mathcal{P} }'$ such that $\mathrm{dim}\,\mathfrak{Der}({\mathcal{P} })<\mathrm{dim}\,\mathfrak{Der}({\mathcal{P} }')$.}


{
Secondly, let ${\mathcal{P} }$ and ${\mathcal{P} }'$ be two  Poisson algebras represented by the structures $(\mu, \mu')$ and $(\lambda, \lambda')$ from $\P_{n}$, respectively. If there exist a parametrized change of basis $g: \mathbb{C}^* \to \textrm{GL}({\mathbb V})$ such that:
$$\lim\limits_{t\to 0} g(t)*(\mu, \mu') = (\lambda, \lambda'),$$
then ${\mathcal{P} }\to {\mathcal{P} }'$. To prove primary degenerations, we will provide the map $g$.}

Thirdly, now to prove primary non-degenerations we will use the following lemma. 

\begin{lemma}\label{main1}
Consider two  Poisson algebras $\mathcal{P}$ and $\mathcal{P}'$. Suppose $\mathcal{P} \to \mathcal{P}'$. Let C be a Zariski closed in $\P_{n}$ that is stable by the action of the invertible upper (lower) triangular matrices. Then if there is a representation $(\mu, \mu')$ of $\mathcal{P}$ in C, then there is a representation $(\lambda, \lambda')$ of $\mathcal{P}'$ in C.
\end{lemma}

\begin{proof}
Suppose $(\mu, \mu')$ is a representation of $\mathcal{P}$ in $C$, then $\overline{O(\mu, \mu')}=\overline{\GL(\mathcal{P})*(\mu, \mu')}$. 
By \cite[Proposition 1.7]{GRH} we have 
\begin{equation}\label{unique}
\overline{\GL(\mathcal{P})*(\mu,\mu')}=\GL(\mathcal{P})*(\overline{T(\mu,\mu')})
\end{equation}
where $T$ is the subgroup of upper (lower) triangular matrices of $\GL(\mathcal{P})$, which is a borel subgroup of $\GL(\mathcal{P})$. Now, taking any 
$(\lambda,\lambda')\in \overline{\GL(\mathcal {P})*(\mu,\mu')}$, formula (\ref{unique}) implies that  $(\lambda,\lambda')=g*(\gamma,\gamma')$ where 
$(\gamma,\gamma')\in\overline{T*(\mu,\mu')}$. But
$(\mu,\mu')\in C$ whence $T*(\mu,\mu')\subset C$ and 
$\overline{T*(\mu,\mu')}\subset C$. So $(\gamma,\gamma')\in C$ and we conclude that
$(\lambda,\lambda')=g*c$ for some $c\in C$. So there is a representative of $(\lambda,\lambda')$ in $C$.
\end{proof}

\begin{remark}

By Lemma \ref{main1}, if we have a polynomial identity in two multiplications {$\mathfrak{P}$}, then it defines a Zariski closed set in {$\P_n$} which is stable up to isomorphism (and in particular, it is stable up to invertible lower triangular matrices). Therefore, any degeneration of a  Poisson algebra that satisfies {$\mathfrak{P}$}, also satisfies {$\mathfrak{P}$}.

\end{remark}

Moreover, we have the following corollary, similar to results found in \cite{bb09, GRH2}.

\begin{corollary}
\label{nondeg}
Consider two  Poisson algebras $\mathcal{P}$ and $\mathcal{P}'$. Suppose $\mathcal{P} \to \mathcal{P}'$. Then if follows:
\begin{enumerate}
    \item $\mathrm{dim}\, \mathrm{Ann} (\mathcal{P}, \cdot)\leq\mathrm{dim}\,\mathrm{Ann}(\mathcal{P}', \cdot)$, 
    \item $\mathrm{dim}\, \mathrm{Ann} (\mathcal{P}, \left\{-,-\right\})\leq\mathrm{dim}\,\mathrm{Ann}(\mathcal{P}', \left\{-,-\right\})$, 
    \item $\mathrm{dim}\, \mathrm{Ann}(\mathcal{P})\leq\mathrm{dim}\,\mathrm{Ann}(\mathcal{P}')$, 
    \item $\mathrm{dim}\, \mathcal{P}\cdot\mathcal{P} \geq \mathrm{dim}\, \mathcal{P}'\cdot\mathcal{P}'$,
    \item $\mathrm{dim}\, \left\{\mathcal{P},\mathcal{P}\right\} \geq \mathrm{dim}\, \left\{\mathcal{P}', \mathcal{P}'\right\}$,
    \item $\mathrm{dim}\, \mathcal{P}^2 \geq \mathrm{dim}\, \mathcal{P}'^2$, 
    
\end{enumerate}
where $\mathrm{Ann}(\mathcal{P})=\left\{x\in \mathcal{P}: x\cdot\mathcal{P}+\left\{x, \mathcal{P}\right\}=0\right\}$ and $\mathcal{P}^2 = \mathcal{P}\cdot\mathcal{P} + \left\{\mathcal{P},\mathcal{P}\right\}$.
\end{corollary}

Let $\mathcal{R}$ be a set of polynomial equations in the variables $c_{ij}^k, c_{ij}'^k$ and in the conditions of the previous result. {Let $C_{\mathcal{R}}$ be the subclass of $C$ of all algebras satisfying the identities of $\mathcal{R}$}. Assume that $(\mu, \mu')\in {C_{\mathcal{R}}}$ and $O((\lambda, \lambda'))\cap {C_{\mathcal{R}}}=\varnothing$, give us the non-degeneration ${\mathcal{P} }\not\to {\mathcal{P} }'$. In this case, we call the identities in $\mathcal{R}$ a separating set for ${\mathcal{P} }\not\to {\mathcal{P} }'$.
To prove non-degenerations, we will present the corresponding separating set and we will omit the verification of the fact that $\mathcal{R}$ is stable under the action of the lower triangular matrices and of the fact that $O((\lambda,\lambda'))\cap \mathcal{R}=\varnothing$, which can be obtained by straightforward calculations using a software like Wolfram.

Finally, if the algebraic classification of the class under consideration is finite, then the graph of primary degenerations gives the whole geometric classification: the description of irreducible components can be easily obtained. However, the variety $\P_{3}$  of $3$-dimensional Poisson algebras contains infinitely many non-isomorphic algebras, since it has the families $\mathcal{P}_{3,4}^{*}$ and $\mathcal{P}_{3,16}^{*}$.

\begin{definition}
\rm
Let ${\mathcal{P} }(*)=\{{\mathcal{P} }(\alpha): {\alpha\in I}\}$ be a family of $n$-dimensional  Poisson algebras and let ${\mathcal{P} }'$ be another  Poisson algebra. Suppose that ${\mathcal{P} }(\alpha)$ is represented by the structure $(\mu(\alpha),\mu'(\alpha))\in\P_{n}$ for $\alpha\in I$ and ${\mathcal{P} }'$ is represented by the structure $(\lambda, \lambda')\in\P_{n}$. We say the family ${\mathcal{P} }(*)$ {degenerates}   to ${\mathcal{P} }'$, and write ${\mathcal{P} }(*)\to {\mathcal{P} }'$, if $(\lambda,\lambda')\in\overline{\{O((\mu(\alpha),\mu'(\alpha)))\}_{\alpha\in I}}$.

Conversely, if $(\lambda,\lambda')\not\in\overline{\{O((\mu(\alpha),\mu'(\alpha)))\}_{\alpha\in I}}$ then we call it a  {non-degeneration}, and we write ${\mathcal{P} }(*)\not\to {\mathcal{P} }'$.

\end{definition}


{On the one hand, to prove ${\mathcal{P} }(*)\to {\mathcal{P} }'$, suppose that ${\mathcal{P} }(\alpha)$ is represented by the structure $(\mu(\alpha),\mu'(\alpha))\in\P_{n}$ for $\alpha\in I$ and ${\mathcal{P} }'$ is represented by the structure $(\lambda, \lambda')\in\P_{n}$. If there exists a pair of maps $(f, g)$, where $f:\mathbb{C}^*\to I$ and $g: \mathbb{C}^* \to \textrm{GL}({\mathbb V})$ are such that:
$$\lim\limits_{t\to 0} g(t)*(\mu\big(f(t)\big),\mu'\big(f(t)\big)) = (\lambda, \lambda'),$$
then ${\mathcal{P} }(*)\to {\mathcal{P} }'$. }

On the other hand, to prove ${\mathcal{P} }(*)\not \to {\mathcal{P} }'$, we will use an analogue of the Lemma \ref{main1} for families of Poisson algebras.

\begin{lemma}\label{main2}
Consider the family of Poisson algebras $\mathcal{P}(*)$ and the Poisson algebra $\mathcal{P}'$. Suppose $\mathcal{P}(*) \to \mathcal{P}'$. Let C be a Zariski closed in $\P_{n}$ that is stable by the action of the invertible upper (lower) triangular matrices. Then if there is a representation $(\mu(\alpha), \mu'(\alpha))$ of $\mathcal{P}(\alpha)$ in C for every $\alpha\in I$, then there is a representation $(\lambda, \lambda')$ of $\mathcal{P}'$ in C.
\end{lemma}

Similarly, constructing a set $\mathcal{R}$ in the conditions of the previous result, such that $(\mu(\alpha), \mu'(\alpha))\in {C_\mathcal{R}}$ for any $\alpha\in I$ and $O((\lambda, \lambda'))\cap {C_\mathcal{R}}=\varnothing$, gives us the non-degeneration ${\mathcal{P} }(*)\not\to {\mathcal{P} }'$. \smallskip

In Theorem \ref{3-dim Poisson} we presented the classification, up to isomorphism, of the complex Poisson algebras of dimension three. To obtain the irreducible components of this variety, we have studied the primary degenerations and non-degenerations first.

\begin{lemma} \label{th:deg4nlts}
The graph of primary degenerations and non-degenerations for the variety of $3$-dimensional Poisson algebras is given in Figure 1, where the numbers on the right side are the dimensions of the corresponding orbits.
\end{lemma}

\bigskip

\begin{proof}
From Table \ref{tab:clas3} we deduce the dimensions of the orbits for each Poisson algebra of dimension $3$. Every primary degeneration and non-degeneration can be proven using the parametrized changes of basis and the separating sets included in Table \ref{tab:algdeg4} and Table \ref{tab:algndeg4} below, respectively. 
Note that, in Table \ref{tab:algdeg4}, a parametrized change of basis $g$ is defined by $g_{i}(t) : = g(t)(e_i) $ for $1\leq i\leq3$, where $\left\{e_i\right\}_1^3$  is the basis from Theorem \ref{3-dim Poisson}.
\end{proof}

\medskip

At this point, only the description of the closure of the orbit of the parametric families $\mathcal{P}_{3,4}^{*}$ and $\mathcal{P}_{3,16}^{*}$ is missing.

\begin{lemma}\label{th:deg4nltsf}
The description of the closure of the orbit of the parametric families $\mathcal{P}_{3,4}^{*}$ and $\mathcal{P}_{3,16}^{*}$ in the variety ${\P_{3}}$ is given.

\begin{enumerate}
    \item The closure of the orbit of the parametric family $\mathcal{P}_{3,4}^{*}$ contains the closures of the orbits of the Poisson algebras $\mathcal{P}_{3,3}, \mathcal{P}_{3,2}$ and $\mathcal{P}_{3,1}$.
    \item The closure of the orbit of the parametric family $\mathcal{P}_{3,16}^{*}$ contains the closures of the orbits of the Poisson algebras $\mathcal{P}_{3,15}, \mathcal{P}_{3,13}, \mathcal{P}_{3,2}$ and $ \mathcal{P}_{3,1}$.
\end{enumerate}

\end{lemma}

\begin{proof}
The primary degenerations and non-degenerations that do not follow from the previous results are included in Table \ref{tab:infseriesdeg} and Table \ref{tab:infseriesndeg}, respectively.
\end{proof}

By Lemma \ref{th:deg4nlts} and Lemma \ref{th:deg4nltsf}, we have the following result that give us the geometric classification of ${\P_{3}}$.

\begin{theorem}
The variety of $3$-dimensional Poisson algebras ${\P_{3}}$ has six irreducible components corresponding to the Poisson algebras $\mathcal{P}_{3,5}, \mathcal{P}_{3,7}, \mathcal{P}_{3,18}$ and $\mathcal{P}_{3,20}$ and the families of  Poisson algebras $\mathcal{P}_{3,4}^{*}$ and $\mathcal{P}_{3,16}^{*}$.
The lattice of subsets for the orbit closures is given in Figure 2, where the numbers above are the dimensions of the corresponding orbits.
\end{theorem}

\section{Degenerations between certain types of Poisson algebras}
\label{sec5}

In this section, as a closing of this work, we will study the degenerations and non-degenerations between the Poisson algebras obtained in Theorem \ref{Poisson with null-filiform} and Theorem  \ref{Poisson with filiform}. This class of algebras is not a variety, although, their study can enrich this classification 
since it will help, for example, to understand how the polynomial identities are inherited between them. 
Obviously, Poisson algebras of dimension three constructed on the commutative associative null-filiforms and filiforms are included in our classification of $3$-dimensional Poisson (see table below). Hence, we may assume $n > 3$.

\begin{table}[H]
\centering
\begin{tabular}{|c|c|c|c|c|c|c|}
\hline
3D & $\mathcal{P}_{3,6}$ & $\mathcal{P}_{3,13}$ & 
$\mathcal{P}_{3,14}$ & $\mathcal{P}_{3,15}$ & $\mathcal{P}_{3,16}^0$ & $\mathcal{P}_{3,16}^{\sqrt{-1}}$  \\ \hline
N/F & $\mathcal{P}_{0}^3$ & $\mathcal{P}_{1,1}^3$ & $\mathcal{P}_{1,2}^3$ & $\mathcal{P}_{1,3}^3$ & $\mathcal{P}_{1,4}^3$ & $\mathcal{P}_{1,5}^3$\\ \hline
\end{tabular}    
\caption{Isomorphisms between algebras in Theorem \ref{3-dim Poisson} and those in Theorem \ref{Poisson with null-filiform} and Theorem \ref{Poisson with filiform}.}
\end{table}

In the first place, to determine the dimension of the orbits of each of these algebras, we compute their algebra of derivations.

\begin{lemma}
\label{derfil0}
Let $\varphi$ be a derivation of $\mathcal{P}_0^n$, then
$\varphi(e_i)= \sum_{k=i}^{n} i \lambda_{k-i+1, 1} e_{k}$ for some $\lambda_{k, 1}\in \mathbb{C}$, $k=1,\ldots, n$. Hence, the dimension of the Lie algebra of derivations of $\mathcal{P}_0^n$ is $n$.

\end{lemma}

\begin{proof}
Clearly, $\varphi(e_1)= \sum_{i=1}^{n} \lambda_{i, 1} e_{i}$ for some $\lambda_{i, 1}\in \mathbb{C}$. Now, for $2\leq i+j\leq n$ we have
$$\varphi(e_{i+j}) = \varphi(e_{i}e_{j}) = \varphi(e_i)e_j + e_{i} \varphi(e_{j}) = (i+j)e_{i+j-1}\varphi(e_{1}).$$
Then, we may fix $k=i+j$ and write
$$\varphi(e_{k}) = k e_{k-1}\varphi(e_{1}) = k e_{k-1} \sum_{i=1}^{n} \lambda_{i, 1} e_{i} = k \sum_{i=1}^{n} \lambda_{i, 1} e_{i + k - 1} = k \sum_{j=k}^{n} \lambda_{j-k+1, 1} e_{j}.$$
The converse is a straightforward verification. Moreover, the Lie algebra of derivations of $\mathcal{P}_0^n$ has dimension $n$. 
\end{proof}

\begin{lemma}
\label{derfil1}
Let $\varphi$ be a derivation of $\mathcal{P}\in \left\{\mathcal{P}_{1,i}^n: i=1,\ldots, 5\right\}$ for $n>3$. 

\begin{itemize}
    \item If $\mathcal{P} = \mathcal{P}_{1,1}^n$, then 
    $$\varphi(e_1) = \sum_{k=1}^{n} \lambda_{k, 1} e_{k}, \quad  
    \varphi(e_i)= \sum_{k=i}^{n-1} i \lambda_{k-i+1, 1} e_{k}, \quad
    \varphi(e_n) = \lambda_{n-1,n} e_{n-1} + \lambda_{n,n} e_{n},$$
    for $2\leq i\leq n -1$, where $\lambda_{k, 1}, \lambda_{n-1,n}, \lambda_{n,n}\in \mathbb{C}$, $k=1,\ldots, n$.
    Moreover, $\textrm{dim}\,\mathfrak{Der}(\mathcal{P}_{1,1}^n) = n+2$ .
    
    \item If $\mathcal{P} = \mathcal{P}_{1,2}^n$, then 
    $$\varphi(e_1) = \sum_{k=2}^{n} \lambda_{k, 1} e_{k}, \quad  
    \varphi(e_i)= \sum_{k=i+1}^{n-1} i \lambda_{k-i+1, 1} e_{k}, \quad
    \varphi(e_n) = \lambda_{n,n} e_{n},$$
    for $2\leq i\leq n -1$, where $\lambda_{k, 1}, \lambda_{n,n}\in \mathbb{C}$, $k=2,\ldots, n$.
    Further, $\textrm{dim}\,\mathfrak{Der}(\mathcal{P}_{1,2}^n) = n$ .
    
    \item If $\mathcal{P} = \mathcal{P}_{1,3}^n$, then 
        $$\varphi(e_1) = \sum_{k=1}^{n} \lambda_{k, 1} e_{k}, \quad  
    \varphi(e_i)= \sum_{k=i}^{n-1} i \lambda_{k-i+1, 1} e_{k}, \quad
    \varphi(e_n) = \lambda_{n-1,n} e_{n-1} + (n-2) \lambda_{1,1} e_{n},$$
    for $2\leq i\leq n -1$, where $\lambda_{k, 1}, \lambda_{n-1,n}\in \mathbb{C}$, $k=1,\ldots, n$.
    Moreover, $\textrm{dim}\,\mathfrak{Der}(\mathcal{P}_{1,3}^n) = n+1$ .
    
    \item If $\mathcal{P} = \mathcal{P}_{1,4}^n$, then 
    $$\varphi(e_1) = \sum_{k=1}^{n} \lambda_{k, 1} e_{k}, \quad  
    \varphi(e_i)= \sum_{k=i}^{n-1} i \lambda_{k-i+1, 1} e_{k}, \quad
    \varphi(e_n) = - \lambda_{n,1} e_{n-1} + \lambda_{n-1,n} e_{n-1} + \frac{n-1}{2} \lambda_{1,1} e_{n},$$
    for $2\leq i\leq n -1$, where $\lambda_{k, 1}, \lambda_{n-1,n}, \lambda_{n,n}\in \mathbb{C}$, $k=1,\ldots, n$.
    Moreover, $\textrm{dim}\,\mathfrak{Der}(\mathcal{P}_{1,4}^n) = n+1$ .
    
    \item If $\mathcal{P} = \mathcal{P}_{1,5}^n$, then 
        $$\varphi(e_1) = \sum_{k=2}^{n} \lambda_{k, 1} e_{k}, \quad  
    \varphi(e_i)= \sum_{k=i+1}^{n-1} i \lambda_{k-i+1, 1} e_{k}, \quad
    \varphi(e_n) = -\lambda_{n,1} e_{n-2} + \lambda_{n-1,n} e_{n-1},$$
    for $2\leq i\leq n -1$, where $\lambda_{k, 1}, \lambda_{n-1,n}\in \mathbb{C}$, $k=2,\ldots, n$.
    Further, $\textrm{dim}\,\mathfrak{Der}(\mathcal{P}_{1,5}^n) = n$ .

\end{itemize}

\end{lemma}

\begin{proof}
 We are proving the case $\mathcal{P} = \mathcal{P}_{1,5}^n$, the rest of the cases can be argued similarly. It is clear that $\varphi(e_1)= \sum_{i=1}^{n} \lambda_{i, 1} e_{i}$ for some $\lambda_{i, 1}\in \mathbb{C}$. In the associative part, for $2\leq k \leq n-1$, we have (see Lemma \ref{derfil0})
 $$\varphi(e_{k}) = k e_{k-1} \sum_{i=1}^{n} \lambda_{i, 1} e_{i} = k \sum_{i=1}^{n-1} \lambda_{i, 1} e_{i + k - 1} = k \sum_{j=k}^{n-1} \lambda_{j-k+1, 1} e_{j}.$$
 Note that $\varphi(e_{n} e_{n}) = 2 \varphi(e_n)e_n$ implies that $\lambda_{n,n} = \frac{(n-1)\lambda_{1,1}}{2}$, where $\varphi(e_n) = \sum_{i=1}^{n} \lambda_{i, n} e_{i}$. For the Lie bracket part, we have 
 $$(n-1)\lambda_{1,1}e_{n-1} = \varphi(e_{n-1}) = \varphi(\left\{ e_{1},e_{n}\right\}) = \left\{\sum_{i=1}^{n} \lambda_{i, 1} e_{i}, e_n\right\} + \left\{ e_1, \varphi(e_n)\right\} = \lambda_{1, 1} e_{n-1} + \frac{(n-1)\lambda_{1,1}}{2} e_{n-1}.$$ 
 Hence, we have $\lambda_{1, 1} = \lambda_{n,n} = 0$.
 
Now, for $2\leq i \leq n$, we have $0 = \varphi(e_{i}e_n) = \varphi(e_{i})e_{n} + e_{i}\varphi(e_{n}) = e_{i}\varphi(e_{n})$. In particular, we obtain $$e_{2}\varphi(e_{n}) = e_{2}\sum_{i=1}^{n-1} \lambda_{i, n} e_{i} = \sum_{i=1}^{n-3} \lambda_{i, n} e_{i+2}= 0, $$
so $\lambda_{i, n} = 0$ for $1\leq i\leq n-3$ and $\varphi(e_n) = \lambda_{n-2,n} e_{n-2} + \lambda_{n-1,n} e_{n-1}$.
Moreover, $0 = \varphi(e_{1}e_n) = \varphi(e_{1})e_{n} + e_{1}\varphi(e_{n}) = \lambda_{n,1} e_{n-1} + \lambda_{n-2,n} e_{n-1}$ implies $\lambda_{n-2,n}=-\lambda_{n,1}$. The converse is a direct verification.  
\end{proof}

\medskip

Now, we study the degenerations and non-degenerations between the algebras under consideration.

\begin{theorem}

The primary degenerations and non-degenerations in the class of Poisson algebras constructed on a  null-filiform or filiform algebra of dimension $n>3$ is given in Figure 3. The numbers on the right side are the dimensions of the orbits.

\end{theorem}

\begin{proof}
The assertions that do not follow by the dimensions of the orbits, which are obtained from Lemma \ref{derfil0} and Lemma \ref{derfil1}, are proved below.

\begin{itemize}

    \item $\mathcal{P}_{1,2}^{n} \to \mathcal{P}_{1,3}^{n}$. Consider the action of the linear map 
    $$g_{i}(t)= t^{-i} e_i, \quad g_{n}(t)= t^{-1}e_{n-1} + e_n,$$
    for $1\leq i\leq n-1$, on $\mathcal{P}_{1,2}^{n}$ to obtain 
    $$\left\{
    \begin{array}{c}
    e_{i}\cdot e_{j}=e_{i+j}, \\ 
    \left\{ e_{1},e_{n}\right\} =e_{n-1} + t e_{n},%
    \end{array}%
    \right. $$
    where $2\leq i+j\leq n-1$. Applying the limit, we have $\mathcal{P}_{1,3}^{n}$.
    
    \item $\mathcal{P}_{1,3}^{n} \to \mathcal{P}_{1,1}^{n}$ and $\mathcal{P}_{1,4}^{n} \to \mathcal{P}_{1,1}^{n}$. Consider the map
    $g_{i}(t)= e_i, \quad g_{n}(t)= t^{-1} e_n,$
    for $1\leq i\leq n-1$. 
    
    The actions $g*\mathcal{P}_{1,3}^{n}$ and $g*\mathcal{P}_{1,4}^{n}$ gives us, respectively, the algebras
    $$\left\{
    \begin{array}{c}
    e_{i}\cdot e_{j}=e_{i+j}, \\ 
    \left\{ e_{1},e_{n}\right\} = t e_{n-1},%
    \end{array}%
    \right. \quad \quad  \quad \quad  \quad \quad 
    e_{i}\cdot e_{j}=e_{i+j}, \quad e_{n}\cdot e_{n} = t^2 e_{n-1}, $$
    where $2\leq i+j\leq n-1$.     Clearly, by taking the limit, we obtain $\mathcal{P}_{1,1}^{n}$ in both cases.

    \item $\mathcal{P}_{1,5}^{n} \to \mathcal{P}_{1,3}^{n}$. Choose, for $1\leq i\leq n-1$, the map 
    $$g_{i}(t)= t^{-i}e_i, \quad g_{n}(t)= t^{-n+2} e_n.$$
    Hence, by applying the action we have
    $$\left\{ 
    \begin{array}{c}
    e_{i}\cdot e_{j}=e_{i+j}, e_{n}\cdot e_{n} = t^{n-3} e_{n-1} \\ 
    \left\{ e_{1},e_{n}\right\} =e_{n-1}.%
    \end{array}%
    \right. $$
    where $2\leq i+j\leq n-1$. Thus, the limit gives us
    $\mathcal{P}_{1,3}^{n}$.
    
    \item $\mathcal{P}_{1,5}^{n} \to \mathcal{P}_{1,4}^{n}$. Consider the map 
    $$g_{i}(t)= t^{2i}e_i, \quad g_{n}(t)= t^{n-1} e_n,$$ 
    for $1\leq i\leq n-1$. Then, the action $g*\mathcal{P}_{1,5}^{n}$ gives us the algebra
    $$\left\{ 
    \begin{array}{c}
    e_{i}\cdot e_{j}=e_{i+j}, e_{n}\cdot e_{n} = e_{n-1} \\ 
    \left\{ e_{1},e_{n}\right\} = t^{n-3} e_{n-1},%
    \end{array}%
    \right. $$
    where $2\leq i+j\leq n-1$. By the limit we obtain
    $\mathcal{P}_{1,4}^{n}$.
   
    \item  $\mathcal{P}_{0}^{n} \to \mathcal{P}_{1,4}^{n}$. Consider, for $1\leq i\leq n-1$, the parametrized isomorphism given by
    $$g_{i}(t)= (-1)^{\frac{-i}{n-1}}t^{-i}e_i, \quad g_{n}(t)= -\sum_{k=2}^{n-1} (-1)^{\frac{-k}{n-1}}t^{n-2k}e_{k}  + t e_{n}.$$
    and its inverse given by
    $$g^{-1}_{i}(t)= (-1)^{\frac{i}{n-1}} t^{i}e_i, \quad g^{-1}_{n}(t)= \sum_{k=2}^{n} t^{n-1-k} e_k.$$
    The verification of  $\lim\limits_{t\to 0}g(t)*\mathcal{P}_{0}^{n} = \mathcal{P}_{1,4}^{n}$ can be studied analogously.

    \item $\mathcal{P}_{0}^{n} \not \to \mathcal{P}_{1,3}^{n}$. This is clear, since $(\mathcal{P}_{0}^{n}, \left\{-,-\right\})$ is the zero algebra and  $(\mathcal{P}_{1,3}^{n}, \left\{-,-\right\})$ is not.

    \item $\mathcal{P}_{1,2}^{n} \not \to \mathcal{P}_{1,4}^{n}$. 
    Observe that $\textrm{Ann}(\mathcal{P}_{1,2}^n, \cdot) = \langle e_{n-1}, e_{n}\rangle$ and  $\textrm{Ann}(\mathcal{P}_{1,4}^n, \cdot) = \langle e_{n-1}\rangle$, then it follows by Corollary \ref{nondeg}.
\end{itemize}
It follows by the dimension of the orbits that every degeneration in Figure 3 is a primary degeneration.
\end{proof}

\newpage

\begin{center}
	
	\begin{tikzpicture}[->,>=stealth,shorten >=0.05cm,auto,node distance=1.3cm,
	thick,
	main node/.style={rectangle,draw,fill=black!30,rounded corners=1.5ex,font=\sffamily \scriptsize \bfseries },
	rigid node/.style={rectangle,draw,fill=black!30,rounded corners=1.5ex,font=\sffamily \scriptsize \bfseries }, 
	poisson node/.style={rectangle,draw,fill=black!35,rounded corners=0ex,font=\sffamily \scriptsize \bfseries },
	ac node/.style={rectangle,draw,fill=white!20,rounded corners=0ex,font=\sffamily \scriptsize \bfseries },
	lie node/.style={rectangle,draw,fill=black!10,rounded corners=0ex,font=\sffamily \scriptsize \bfseries },
	style={draw,font=\sffamily \scriptsize \bfseries }]

	\node (13) at (0.5,12) {$9$};
	\node (12) at (0.5,11) {$8$};
	\node (11) at (0.5,9) {$7$};
	\node (10) at (0.5,7) {$6$};
	\node (9)  at (0.5,5) {$5$};
	\node (8)  at (0.5,3) {$4$};
	\node (6)  at (0.5,2) {$3$};
	\node (0)  at (0.5,0) {$0$};

    
    \node[poisson node] (c31) at (-8,0) {$\mathcal{P}_{3,1}$};
    
    \node[lie node] (c32) at (-12,2) {$\mathcal{P}_{3,2}$};
    \node[lie node] (c34a1) at (-10,2) {$\mathcal{P}_{3,4}^{1}$};
    
    \node[ac node] (c313) at (-6,3) {$\mathcal{P}_{3,13}$};
    
    \node[lie node] (c33) at (-14,5) {$\mathcal{P}_{3,3}$};
    \node[lie node] (c34a) at (-12,5) {$\mathcal{P}_{3,4}^{\alpha\neq1}$};
    \node[poisson node] (c315) at (-10,5) {$\mathcal{P}_{3,15}$};
    \node[poisson node] (c316a) at (-8,5) {$\mathcal{P}_{3,16}^{\beta\neq0}$};
    \node[ac node] (c316a0) at (-4,5) {$\mathcal{P}_{3,16}^{0}$};    
    \node[ac node] (c317) at (-6,5) {$\mathcal{P}_{3,17}$};
    \node[ac node] (c319) at (-2,5) {$\mathcal{P}_{3,19}$};
    
    \node[lie node] (c35) at (-13,7) {$\mathcal{P}_{3,5}$};
    \node[ac node] (c36) at (-4,7) {$\mathcal{P}_{3,6}$};
    \node[poisson node] (c314) at (-9,7) {$\mathcal{P}_{3,14}$};
    
    \node[ac node] (c39) at (-6,9) {$\mathcal{P}_{3,9}$};
    \node[ac node] (c311) at (-4,9) {$\mathcal{P}_{3,11}$};
    \node[ac node] (c312) at (-2,9) {$\mathcal{P}_{3,12}$};
    \node[poisson node] (c318) at (-10,9) {$\mathcal{P}_{3,18}$};
    \node[poisson node] (c320) at (-8,9) {$\mathcal{P}_{3,20}$};
    
    \node[ac node] (c38) at (-6,11) {$\mathcal{P}_{3,8}$};
    \node[ac node] (c310) at (-2,11) {$\mathcal{P}_{3,10}$};
    
    \node[ac node] (c37) at (-4,12) {$\mathcal{P}_{3,7}$};
 
	\path[every node/.style={font=\sffamily\small}]


    (c318) edge   node {} (c314)
    (c318) edge   node {} (c317)
    (c320) edge   node {} (c314)
    (c320) edge   node {} (c319)

    (c314) edge   node {} (c315)
    (c314) edge   node[above=5, right=-8, fill=white]{\tiny $\alpha= 0$ } (c34a)

    (c37) edge   node {} (c38)
    (c37) edge   node {} (c310)
    
    (c38) edge   node {} (c312)
    (c38) edge   node {} (c311)
    (c38) edge   node {} (c39)
    (c310) edge   node {} (c311)
    (c310) edge   node {} (c312)
    
    (c312) edge   node {} (c36)
    (c312) edge   node {} (c319)
    (c311) edge   node {} (c36)
    (c39) edge   node {} (c317)
    (c39) edge   node {} (c36)
    
    (c317) edge   node {} (c313)
    (c319) edge   node {} (c313)
    
    (c35) edge   node[above=5, right=-16, fill=white]{\tiny $\alpha= -1$ } (c34a)
    (c34a) edge   node {} (c32)
    (c33) edge   node {} (c32)
    (c33) edge   node {} (c34a1)
    (c36) edge   node {} (c316a0)
    (c316a) edge   node {} (c313)
    (c316a0) edge   node {} (c313)
    (c315) edge   node {} (c313)
    (c315) edge   node {} (c32)
    
    (c32) edge   node {} (c31)
    (c34a1) edge   node {} (c31)
    (c313) edge   node {} (c31);
    
	\end{tikzpicture}
{\tiny \begin{itemize}	
Legend:
\begin{itemize}
    \item[--] Squared nodes: Poisson algebras.
    \item[--] Squared white nodes: embedded commutative associative algebras $(\mu, 0)$.
    \item[--] Squared grey nodes: embedded Lie algebras $(0, \mu)$.
\end{itemize}
\end{itemize}}

\label{fig1}
{Figure 1.}  Graph of primary degenerations and non-degenerations.	
\end{center}

\begin{center}\label{lattice}
{\tiny
\begin{tikzpicture}[-,draw=black!100,node distance=0.73cm,
                   thick,main node/.style={rectangle, fill=gray!30,font=\sffamily \scriptsize \bfseries },style={draw,font=\sffamily \scriptsize \bfseries }]

\node (8) {$9$};
\node (8r) [right  of=8] {};
\node (8rr) [right  of=8r] {};
\node (8rrr) [right  of=8rr] {};
\node (7) [right  of=8rr]      {$8$};
\node (7r) [right  of=7] {};
\node (7rr) [right  of=7r] {};
\node (7rrr) [right  of=7rr] {};
\node (6) [right  of=7rr]      {$7$};
\node (6r) [right  of=6] {};
\node (6rr) [right  of=6r] {};
\node (6rrr) [right  of=6rr] {};
\node (5) [right  of=6rr]      {$6$};
\node (5r) [right  of=5] {};
\node (5rr) [right  of=5r] {};
\node (5rrr) [right  of=5rr] {};
\node (4) [right  of=5rrr]      {$5$};
\node (4r) [right  of=4] {};
\node (4rr) [right  of=4r] {};
\node (4rrr) [right  of=4rr] {};
\node (3) [right  of=4rrr]      {$4$};
\node (3r) [right  of=3] {};
\node (3rr) [right  of=3r] {};
\node (2) [right  of=3rr]      {$3$};
\node (2r) [right  of=2] {};
\node (2rr) [right  of=2r] {};
\node (2rrr) [right  of=2rr] {};
\node (1) [right  of=2rr]      {$0$};
\node (1r) [right  of=1] {};
\node (1rr) [right  of=1r] {};
\node (1rrr) [right  of=1rr] {};

    \node (13p1)  [below of =8]     	{};
	\node (13p2)  [below of =13p1]     	{};
	\node (13p3)  [below of =13p2]     	{};
    \node (13p4)  [below of =13p3]     	{};
    \node (13p5)  [below of =13p4]     	{};
    \node (13p6)  [below of =13p5]     	{};
    \node (13p7)  [below of =13p6]     	{};
    \node (13p8)  [below of =13p7]     	{};
    \node (13p9)  [below of =13p8]     	{};
    
	\node[main node] (c37)  [below of =13p1]     	{$\overline{O\big(\mathcal{P}_{3,7}\big)}$};	
    
    \node (12p1)  [below of =7]     	{};
	\node (12p2)  [below of =12p1]     	{};
	\node (12p3)  [below of =12p2]     	{};
    \node (12p4)  [below of =12p3]     	{};
	\node (12p5)  [below of =12p4]     	{};
	\node (12p6)  [below of =12p5]     	{};
	\node (12p7)  [below of =12p6]     	{};
	\node (12p8)  [below of =12p7]     	{};
	\node (12p9)  [below of =12p8]     	{};

	\node[main node] (c310)  [below of =7]     	{$\overline{O\big(\mathcal{P}_{3,10}\big)}$};	
	\node[main node] (c38)  [below of =12p2]     	{$\overline{O\big(\mathcal{P}_{3,8}\big)}$};

	\node (11p1)  [below of =6]     	{};
	\node (11p2)  [below of =11p1]     	{};
	\node (11p3)  [below of =11p2]     	{};
	\node (11p4)  [below of =11p3]     	{};
	\node (11p5)  [below of =11p4]     	{};
	\node (11p6)  [below of =11p5]     	{};
	\node (11p7)  [below of =11p6]     	{};
	\node (11p8)  [below of =11p7]     	{};
	\node (11p9)  [below of =11p8]     	{};
	
	\node[main node] (c312)  [below of =6]     	{$\overline{O\big(\mathcal{P}_{3,12}\big)}$};	
	\node[main node] (c311)  [below of =11p1]     	{$\overline{O\big(\mathcal{P}_{3,11}\big)}$};	
	\node[main node] (c39)  [below of =11p2]     	{$\overline{O\big(\mathcal{P}_{3,9}\big)}$};
	\node[main node] (c320)  [below of =11p5]     	{$\overline{O\big(\mathcal{P}_{3,20}\big)}$};
	\node[main node] (c318)  [below of =11p7]     	{$\overline{O\big(\mathcal{P}_{3,18}\big)}$};

	\node (10p1)  [below of =5]     	{};
	\node (10p2)  [below of =10p1]     	{};
	\node (10p3)  [below of =10p2]     	{};
	\node (10p4)  [below of =10p3]     	{};
	\node (10p5)  [below of =10p4]     	{};
	\node (10p6)  [below of =10p5]     	{};
	\node (10p7)  [below of =10p6]     	{};
	\node (10p8)  [below of =10p7]     	{};
	\node (10p9)  [below of =10p8]     	{};

	\node[main node] (c36)  [below of =10p1]     	{$\overline{O\big(\mathcal{P}_{3,6}\big)}$};	
	\node[main node] (c316*)  [below of =10p4]     	{$\overline{O\big(\mathcal{P}_{3,16}^{*}\big)}$};	
	\node[main node] (c314)  [below of =10p7]     	{$\overline{O\big(\mathcal{P}_{3,14}\big)}$};	
	\node[main node] (c34*)  [below of =10p6]     	{$\overline{O\big(\mathcal{P}_{3,4}^{*}\big)}$};	
	\node[main node] (c35)  [below of =10p8]     	{$\overline{O\big(\mathcal{P}_{3,5}\big)}$};
	
	\node (9p1)  [below of =4]     	{};
	\node (9p2)  [below of =9p1]     	{};
	\node (9p3)  [below of =9p2]     	{};
	\node (9p4)  [below of =9p3]     	{};
	\node (9p5)  [below of =9p4]     	{};
	\node (9p6)  [below of =9p5]     	{};
	\node (9p7)  [below of =9p6]     	{};
	\node (9p8)  [below of =9p7]     	{};
	\node (9p9)  [below of =9p8]     	{};
	
	\node[main node] (c319)  [below of =4]     	{$\overline{O\big(\mathcal{P}_{3,19}\big)}$};	
	\node[main node] (c316a0)  [below of =9p1]     	{$\overline{O\big(\mathcal{P}_{3,16}^{0}\big)}$};
	\node[main node] (c317)  [below of =9p2]     	{$\overline{O\big(\mathcal{P}_{3,17}\big)}$};	
	\node[main node] (c315)  [below of =9p5]     	{$\overline{O\big(\mathcal{P}_{3,15}\big)}$};
	\node[main node] (c34a0)  [below of =9p6]     	{$\overline{O\big(\mathcal{P}_{3,4}^{0}\big)}$};	
	\node[main node] (c34a-1)  [below of =9p7]     	{$\overline{O\big(\mathcal{P}_{3,4}^{-1}\big)}$};
	\node[main node] (c33)  [below of =9p8]     	{$\overline{O\big(\mathcal{P}_{3,3}\big)}$};

	\node (8p1)  [below of =3]     	{};
	\node (8p2)  [below of =8p1]     	{};
	\node (8p3)  [below of =8p2]     	{};
	\node (8p4)  [below of =8p3]     	{};
	\node (8p5)  [below of =8p4]     	{};
	\node (8p6)  [below of =8p5]     	{};
	\node (8p7)  [below of =8p6]     	{};
	\node (8p8)  [below of =8p7]     	{};
	\node (8p9)  [below of =8p8]     	{};
		
	\node[main node] (c313)  [below of =8p2]     	{$\overline{O\big(\mathcal{P}_{3,13}\big)}$};
	
	\node (5p1)  [below of =2]     	{};
	\node (5p2)  [below of =5p1]     	{};
	\node (5p3)  [below of =5p2]     	{};
	\node (5p4)  [below of =5p3]     	{};
	\node (5p5)  [below of =5p4]     	{};
	\node (5p6)  [below of =5p5]     	{};
	\node (5p7)  [below of =5p6]     	{};
	\node (5p8)  [below of =5p7]     	{};
	\node (5p9)  [below of =5p8]     	{};
	
	\node[main node] (c32)  [below of =5p6]     	{$\overline{O\big(\mathcal{P}_{3,2}\big)}$};
	
	\node[main node] (c34a1)  [below of =5p8]     	{$\overline{O\big(\mathcal{P}_{3,4}^{1}\big)}$};
	
	\node (0p1)  [below of =1]     	{};
	\node (0p2)  [below of =0p1]     	{};
	\node (0p3)  [below of =0p2]     	{};
	\node (0p4)  [below of =0p3]     	{};
	\node (0p5)  [below of =0p4]     	{};
	\node (0p6)  [below of =0p5]     	{};
	\node (0p7)  [below of =0p6]     	{};
	\node (0p8)  [below of =0p7]     	{};
	\node (0p9)  [below of =0p8]     	{};
	\node[main node] (c31) [ below  of =0p4]       {$\overline{O\big(\mathcal{P}_{3,1}\big)}$};
	
\path
    (c37) edge   (c38)
    (c37) edge   (c310)    
    (c38) edge   (c39) 
    (c38) edge   (c311) 
    (c38) edge   (c312) 
    (c310) edge   (c311) 
    (c310) edge   (c312) 
    (c39) edge   (c36) 
    (c39) edge   (c317) 
    (c311) edge   (c36) 
    (c312) edge   (c36) 
    (c312) edge   (c319)
    (c36) edge   (c316a0)
    (c317) edge   (c313) 
    (c316a0) edge   (c313) 
    (c319) edge   (c313) 
    
    (c35) edge   (c34a-1)
    (c314) edge   (c315)
    (c314) edge   (c34a0)
    
    (c34*) edge   (c34a-1)
    (c34*) edge   (c34a0)
    (c34*) edge   (c33)
    (c316*) edge   (c316a0)
    (c316*) edge   (c313)
    (c315) edge   (c313)
    (c316*) edge   (c315)
    (c315) edge   (c32)
    (c318) edge   (c314) 
    (c318) edge   (c317) 
    (c320) edge   (c314) 
    (c320) edge   (c319) 
    (c34a-1) edge   (c32)
    (c34a0) edge   (c32)
    (c33) edge   (c32)
    (c33) edge   (c34a1)
    (c313) edge   (c31)
    (c34a1) edge   (c31)
    (c32) edge   (c31);

\end{tikzpicture}}

\label{fig2}
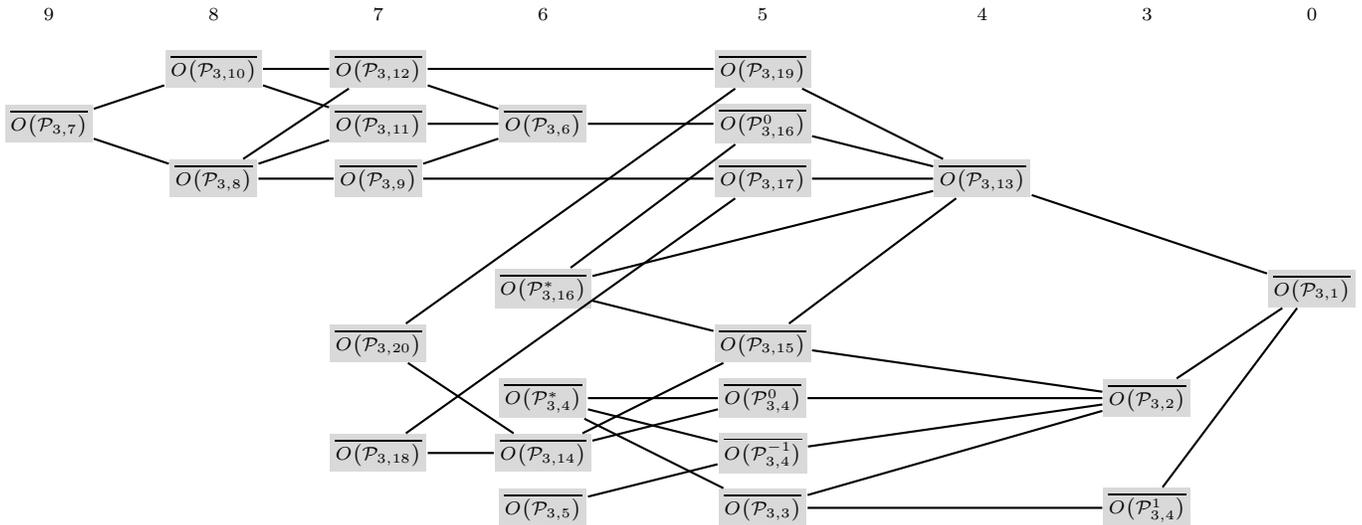
{Figure 2.}  Inclusion graph for orbit closures and dimension. 
\end{center}

\begin{table}[H]
    \centering
    \begin{tabular}{|l|l|l|l|l|}
			\hline
			${\mathcal{P} }$ & Multiplication table & $\mathrm{dim}\,\mathfrak{Der}$& $(\mathcal{P}, \cdot)$ & $(\mathcal{P}, \left\{-,-\right\})$ \ \\
			\hline
            \hline

			$\mathcal{P}_{3,1}$ & & $9$ & $A_{1}$  & $\mathcal{L}_{3,1}$	\\
			\hline
			
            $\mathcal{P}_{3,2}$ & $\left\{ e_{1},e_{2}\right\}             =e_{3}.$ & $6 $ & $A_{1}$ & $\mathcal{L}_{3,2}$
     
            \\ \hline $\mathcal{P}_{3,3}$ & $\left\{ e_{1},e_{2}\right\}             =e_{2},\left\{
            e_{1},e_{3}\right\} =e_{2}+e_{3}.$ & $4 $ & $A_{1}$ & $\mathcal{L}_{3,3}$
     
            \\ \hline $\mathcal{P}_{3,4}^{1}$ & $\left\{               e_{1},e_{2}\right\}
            =e_{2},\left\{ e_{1},e_{3}\right\} = e_{3}.$ & $ 6$ & $A_{1}$ & $\mathcal{L}_{3,4}^{1}$
            
            \\ \hline $\mathcal{P}_{3,4}^{\alpha\neq1}$ & $\left\{               e_{1},e_{2}\right\}
            =e_{2},\left\{ e_{1},e_{3}\right\} =\alpha e_{3}.$ & $4 $ & $A_{1}$ & $\mathcal{L}_{3,4}^{\alpha}$
            
            \\ \hline $\mathcal{P}_{3,5}$ & $\left\{ e_{1},e_{2}\right\}             =e_{3},\left\{
            e_{1},e_{3}\right\} =-2e_{1},\left\{ e_{2},e_{3}\right\}             =2e_{2}.$ & $3 $ & $A_{1}$ & $\mathcal{L}_{3,5}$
     
            \\ \hline $\mathcal{P}_{3,6}$ & $e_{1}\cdot e_{1}=e_{2},e_{1}\cdot          e_{2}=e_{3}.$ & $3 $ & $A_{4}$ & $\mathcal{L}_{3,1}$
     
            \\ \hline $\mathcal{P}_{3,7}$ & $e_{1}\cdot e_{1}=e_{1},e_{2}\cdot
            e_{2}=e_{2},e_{3}\cdot e_{3}=e_{3}.$ & $0 $ &$A_{5}$ & $\mathcal{L}_{3,1}$
     
            \\ \hline $\mathcal{P}_{3,8}$ & $e_{1}\cdot e_{1}=e_{1},e_{2}\cdot
            e_{2}=e_{2},e_{2}\cdot e_{3}=e_{3}.$ & $1 $ &$A_{6}$ & $\mathcal{L}_{3,1}$
     
            \\ \hline $\mathcal{P}_{3,9}$ & $e_{1}\cdot e_{1}=e_{1},e_{1}\cdot
            e_{2}=e_{2},e_{1}\cdot e_{3}=e_{3},e_{2}\cdot e_{2}=e_{3}.$ & $2 $ &$A_{8}$ & $\mathcal{L}_{3,1}$
     
            \\ \hline $\mathcal{P}_{3,10}$ & $e_{1}\cdot e_{1}=e_{1},e_{2}\cdot          e_{2}=e_{2}.$ & $1 $ & $A_{9}$& $\mathcal{L}_{3,1}$
     
            \\ \hline $\mathcal{P}_{3,11}$ & $e_{1}\cdot e_{1}=e_{1},e_{1}\cdot          e_{2}=e_{2}.$ & $2 $ & $A_{10}$& $\mathcal{L}_{3,1}$
     
            \\ \hline $\mathcal{P}_{3,12}$ & $e_{1}\cdot e_{1}=e_{1},e_{2}\cdot          e_{2}=e_{3}.$ & $2 $ &$A_{12}$ & $\mathcal{L}_{3,1}$
     
            \\ \hline $\mathcal{P}_{3,13}$ & $e_{1}\cdot e_{1}=e_{2}.$ & $5 $ &$A_{2}$ & $\mathcal{L}_{3,1}$
     
            \\ \hline $\mathcal{P}_{3,14}$ & $\left\{ 
            \begin{tabular}{l}
            $e_{1}\cdot e_{1}=e_{2},$ \\ 
            $\left\{ e_{1},e_{3}\right\} =e_{3}.$%
            \end{tabular}%
            \right. $ & $3 $ &$A_{2}$ & $\mathcal{L}_{3,4}^{0}$
     
            \\ \hline $\mathcal{P}_{3,15}$ & $\left\{ 
            \begin{tabular}{l}
            $e_{1}\cdot e_{1}=e_{2},$ \\ 
            $\left\{ e_{1},e_{3}\right\} =e_{2}.$%
            \end{tabular}%
            \right. $ & $4 $ &$A_{2}$ & $\mathcal{L}_{3,2}$
     
            \\ \hline $\mathcal{P}_{3,16}^{\alpha }$ & $\left\{ 
            \begin{tabular}{l}
            $e_{1}\cdot e_{2}=e_{3},$ \\ 
            $\left\{ e_{1},e_{2}\right\} =\alpha e_{3}.$%
            \end{tabular}%
            \right. $ & $4 $ &$A_{3}$ & $\mathcal{L}_{3,2}$
     
            \\ \hline $\mathcal{P}_{3,17}$ & $e_{1}\cdot e_{1}=e_{1},e_{1}\cdot
            e_{2}=e_{2},e_{1}\cdot e_{3}=e_{3}.$ & $4 $ &$A_{7}$ & $\mathcal{L}_{3,1}$
     
            \\ \hline $\mathcal{P}_{3,18}$ & $\left\{ 
            \begin{tabular}{l}
            $e_{1}\cdot e_{1}=e_{1},e_{1}\cdot e_{2}=e_{2},e_{1}\cdot          e_{3}=e_{3},$ \\ 
            $\left\{ e_{2},e_{3}\right\} =e_{2}.$%
            \end{tabular}%
            \right. $ & $2 $ &$A_{7}$ & $\mathcal{L}_{3,4}^{0}$
     
            \\ \hline $\mathcal{P}_{3,19}$ & $e_{1}\cdot e_{1}=e_{1}.$ & $4 $ &$A_{11}$ & $\mathcal{L}_{3,1}$
     
            \\ \hline $\mathcal{P}_{3,20}$ & $\left\{ 
            \begin{tabular}{l}
            $e_{1}\cdot e_{1}=e_{1},$ \\ 
            $\left\{ e_{2},e_{3}\right\} =e_{2}.$%
            \end{tabular}%
            \right. $ & $2 $ &$A_{11}$ & $\mathcal{L}_{3,4}^{0}$
            \\ \hline
    \end{tabular}
    \caption{The classification of the $3$-dimensional Poisson algebras.}
    \label{tab:clas3}
\end{table}

{\tiny
\begin{table}[H]
    \centering
    \begin{tabular}{|rcl|lll|}
\hline
\multicolumn{3}{|c|}{\textrm{Degeneration}}  & \multicolumn{3}{|c|}{\textrm{Parametrized change of basis}} \\
\hline
\hline

$\mathcal{P}_{3,17} $ & $\to$ & $\mathcal{P}_{3,13}  $ & $
g_{1}(t)= t^{-1}e_1 - t^{-2}e_2,$ & $
g_{2}(t)= e_2,$ & $
g_{3}(t)= e_3.$ \\
\hline

$\mathcal{P}_{3,19} $ & $\to$ & $\mathcal{P}_{3,13}  $ & $
g_{1}(t)= t^{-1}e_1+t^{-2}e_2,$ & $
g_{2}(t)= e_2,$ & $
g_{3}(t)= e_3.$ \\
\hline

$\mathcal{P}_{3,9} $ & $\to$ & $\mathcal{P}_{3,17}  $ & $
g_{1}(t)= e_1,$ & $
g_{2}(t)= t^{-1}e_2,$ & $
g_{3}(t)= e_3.$ \\
\hline

$\mathcal{P}_{3,9} $ & $\to$ & $\mathcal{P}_{3,6}  $ & $
g_{1}(t)= t^{-1}e_1-t^{-2}e_2+t^{-3}e_3,$ & $
g_{2}(t)= e_2,$ & $
g_{3}(t)= te_3.$ \\
\hline

$\mathcal{P}_{3,11} $ & $\to$ & $\mathcal{P}_{3,6}  $ & $
g_{1}(t)= t^{-1}e_1-t^{-2}e_2-t^{-3}e_3,$ & $
g_{2}(t)= e_2+t^{-1}e_3,$ & $
g_{3}(t)= e_3.$ \\
\hline

$\mathcal{P}_{3,12} $ & $\to$ & $\mathcal{P}_{3,6}  $ & $
g_{1}(t)= \frac{1}{2}t^{-1}e_1+\frac{1}{2}t^{-2}e_2+t^{-3}e_3,$ & $
g_{2}(t)= -e_1+4t^{-2}e_3,$ & $
g_{3}(t)= 2e_2-4t^{-1}e_3.$ \\
\hline

$\mathcal{P}_{3,12} $ & $\to$ & $\mathcal{P}_{3,19}  $ & $
g_{1}(t)= e_1,$ & $
g_{2}(t)= t^{-1}e_2,$ & $
g_{3}(t)= e_3.$ \\
\hline

$\mathcal{P}_{3,10} $ & $\to$ & $\mathcal{P}_{3,11}  $ & $
g_{1}(t)= e_1-t^{-1}e_2,$ & $
g_{2}(t)= t^{-1}e_2,$ & $
g_{3}(t)= e_3.$ \\
\hline

$\mathcal{P}_{3,10} $ & $\to$ & $\mathcal{P}_{3,12}  $ & $
g_{1}(t)= e_1,$ & $
g_{2}(t)= t^{-1}e_2+t^{-2}e_3,$ & $
g_{3}(t)= e_3.$ \\
\hline

$\mathcal{P}_{3,8} $ & $\to$ & $\mathcal{P}_{3,12}  $ & $
g_{1}(t)= e_1,$ & $
g_{2}(t)= t^{-1}e_2-t^{-2}e_3,$ & $
g_{3}(t)= e_3.$ \\
\hline

$\mathcal{P}_{3,8} $ & $\to$ & $\mathcal{P}_{3,11}  $ & $
g_{1}(t)= t^{-1}e_3,$ & $
g_{2}(t)= e_1,$ & $
g_{3}(t)= e_2.$ \\
\hline

$\mathcal{P}_{3,8} $ & $\to$ & $\mathcal{P}_{3,9}  $ & $
g_{1}(t)= e_1+t^{-1}e_2+t^{-2}e_3,$ & $
g_{2}(t)= -t^{-1}e_2-t^{-2}e_3,$ & $
g_{3}(t)= e_3.$ \\
\hline

$\mathcal{P}_{3,7} $ & $\to$ & $\mathcal{P}_{3,8}  $ & $
g_{1}(t)= e_1,$ & $
g_{2}(t)= e_2-t^{-1}e_3,$ & $
g_{3}(t)= t^{-1}e_3.$ \\
\hline

$\mathcal{P}_{3,7} $ & $\to$ & $\mathcal{P}_{3,10}  $ & $
g_{1}(t)= e_1,$ & $
g_{2}(t)= e_2,$ & $
g_{3}(t)= t^{-1}e_3.$ \\
\hline

$\mathcal{P}_{3,3} $ & $\to$ & $\mathcal{P}_{3,2}  $ & $
g_{1}(t)= t^{-1}e_1,$ & $
g_{2}(t)= e_2+t^{-1}e_3,$ & $
g_{3}(t)= -t^{-1}e_3.$ \\
\hline

$\mathcal{P}_{3,3} $ & $\to$ & $\mathcal{P}_{3,4}^{1}  $ & $
g_{1}(t)= e_1,$ & $
g_{2}(t)= te_2,$ & $
g_{3}(t)= e_3.$ \\
\hline

$\mathcal{P}_{3,4}^{\alpha\neq1} $ & $\to$ & $\mathcal{P}_{3,2}  $ & $
g_{1}(t)= t^{-1}e_1,$ & $
g_{2}(t)= e_2-t^{-1}(\alpha-1)^{-1}e_3,$ & $
g_{3}(t)= e_3.$ \\
\hline

$\mathcal{P}_{3,5} $ & $\to$ & $\mathcal{P}_{3,4}^{-1}  $ & $
g_{1}(t)= -2t^{2}e_1+te_2,$ & $
g_{2}(t)= -\frac{1}{2}t^{-2}e_1-\frac{1}{4}t^{-3}e_2,$ & $
g_{3}(t)= t^{-1}e_2+e_3.$ \\
\hline

$\mathcal{P}_{3,18} $ & $\to$ & $\mathcal{P}_{3,14}  $ & $
g_{1}(t)= t^{-1}e_1-t^{-2}e_2,$ & $
g_{2}(t)= te_2+e_3,$ & $
g_{3}(t)= t^{-2}e_3.$ \\
\hline

$\mathcal{P}_{3,20} $ & $\to$ & $\mathcal{P}_{3,14}  $ & $
g_{1}(t)= t^{-1}e_1+t^{-2}e_2,$ & $
g_{2}(t)= -te_2+e_3,$ & $
g_{3}(t)= t^{-2}e_3.$ \\
\hline

$\mathcal{P}_{3,14} $ & $\to$ & $\mathcal{P}_{3,15}  $ & $
g_{1}(t)= t^{-1}e_1,$ & $
g_{2}(t)= t^{-2}e_2-t^{-1}e_3,$ & $
g_{3}(t)= t^{-1}e_2.$ \\
\hline

$\mathcal{P}_{3,14} $ & $\to$ & $\mathcal{P}_{3,4}^{0}  $ & $
g_{1}(t)= e_1,$ & $
g_{2}(t)= te_3,$ & $
g_{3}(t)= e_2.$ \\
\hline

$\mathcal{P}_{3,20} $ & $\to$ & $\mathcal{P}_{3,19}  $ & $
g_{1}(t)= e_1,$ & $
g_{2}(t)= e_2,$ & $
g_{3}(t)= t^{-1}e_3.$ \\
\hline

$\mathcal{P}_{3,18} $ & $\to$ & $\mathcal{P}_{3,17}  $ & $
g_{1}(t)= e_1,$ & $
g_{2}(t)= e_2,$ & $
g_{3}(t)= t^{-1}e_3.$ \\
\hline

$\mathcal{P}_{3,15} $ & $\to$ & $\mathcal{P}_{3,2}  $ & $
g_{1}(t)= t^{-1}e_1,$ & $
g_{2}(t)= te_2 + e_3,$ & $
g_{3}(t)= -e_3.$ \\
\hline

$\mathcal{P}_{3,15} $ & $\to$ & $\mathcal{P}_{3,13}  $ & $
g_{1}(t)= te_1,$ & $
g_{2}(t)= t^{2}e_2,$ & $
g_{3}(t)= e_3.$ \\
\hline

$\mathcal{P}_{3,6} $ & $\to$ & $\mathcal{P}_{3,16}^{0}  $ & $
g_{1}(t)= e_1,$ & $
g_{2}(t)= te_2,$ & $
g_{3}(t)= te_3.$ \\
\hline

$\mathcal{P}_{3,16}^{\alpha} $ & $\to$ & $\mathcal{P}_{3,13}  $ & $
g_{1}(t)= t^2e_1+e_2,$ & $
g_{2}(t)= t^{-4}e_2,$ & $
g_{3}(t)= -\frac{1}{2}e_2+te_3.$ \\
\hline

    \end{tabular}
    \caption{Degenerations of $3$-dimensional Poisson algebras.}
    \label{tab:algdeg4}
\end{table}}

 \begin{table}[H]
    \centering
    \begin{tabular}{|l|l|}
\hline
\multicolumn{1}{|c|}{\textrm{Non-degeneration}} & \multicolumn{1}{|c|}{\textrm{Arguments}}\\
\hline
\hline
$\begin{array}{cccc}
     \mathcal{P}_{3,5} &\not \to&  \mathcal{P}_{3,3}, \mathcal{P}_{3,4}^{\alpha}(\alpha\neq-1) &  
\end{array}$ &
${\mathcal R}= \left\{ \begin{array}{l}
c_{21}'^1=-c_{12}'^1, c_{21}'^2=-c_{12}'^2, c_{21}'^3=-c_{12}'^3, c_{31}'^1=-c_{13}'^1 \\
c_{31}'^2=-c_{13}'^2, c_{31}'^3=-c_{13}'^3=c_{12}^2, c_{32}'^2=-c_{23}'^2=c_{13}^1, \\
c_{32}'^3=-c_{23}'^3=-c_{12}^1, c_{ij}^{k}=c_{ij}'^{k}=0 \textrm{ otherwise} 
\end{array} \right\}$\\
\hline
$\begin{array}{cccc}
     \mathcal{P}_{3,9} &\not \to&\mathcal{P}_{3,19} &  
\end{array}$ &
${\mathcal R}= \left\{ \begin{array}{l}
c_{11}^{1}, c_{11}^{2}, c_{11}^{3}, c_{22}^{3}\in \mathbb{C},  c_{21}^{2}=c_{12}^{2}=c_{13}^{3}=c_{31}^{3}=c_{11}^{1},\\
c_{21}^{3}=c_{12}^{3}, c_{11}^{1}c_{12}^{3}=c_{22}^{3}c_{11}^{2}, c_{ij}^{k}=c_{ij}'^{k}=0 \textrm{ otherwise} 
\end{array} \right\}$\\
\hline
$\begin{array}{cccc}
     \mathcal{P}_{3,11} &\not \to&  
     \mathcal{P}_{3,19} &  
\end{array}$ &
${\mathcal R}= \left\{ \begin{array}{l}
c_{11}^{2},c_{11}^3\in\mathbb{C}, c_{21}^2=c_{12}^2=c_{11}^1, c_{21}^{3}=c_{12}^{3},\\
c_{ij}^{k}=c_{ij}'^{k}=0 \textrm{ otherwise} 
\end{array} \right\}$\\
\hline
$\begin{array}{cccc}
     \mathcal{P}_{3,10} &\not \to&  
     \mathcal{P}_{3,17} &  
\end{array}$ &
${\mathcal R}= \left\{ \begin{array}{l}
c_{11}^{1}, c_{11}^{2}, c_{11}^{3} \in \mathbb{C},c_{21}^2=c_{12}^2, c_{21}^{3}=c_{12}^3,\\ c_{12}^2c_{22}^3=c_{12}^3c_{22}^2,
c_{ij}^{k}=c_{ij}'^{k}=0 \textrm{ otherwise} 
\end{array} \right\}$\\
\hline
$\begin{array}{cccc}
     \mathcal{P}_{3,16}^{\alpha} &\not \to&  \mathcal{P}_{3,2}, \mathcal{P}_{3,4}^{1}&  
\end{array}$ &
${\mathcal R}= \left\{ \begin{array}{l}
c_{11}^3\in\mathbb{C}, c_{21}^3=c_{12}^3, c_{21}'^3=-c_{12}'^3=-\alpha c_{12}^3\\
c_{ij}^{k}=c_{ij}'^{k}=0 \textrm{ otherwise} 
\end{array} \right\}$\\
\hline
$\begin{array}{cccc}
     \mathcal{P}_{3,18} &\not \to&  
\mathcal{P}_{3,4}^{\alpha}(\alpha\neq0),
\mathcal{P}_{3,16}^{\alpha}, \mathcal{P}_{3,19}&  
\end{array}$ &
${\mathcal R}= \left\{ \begin{array}{l}
c_{11}^1, c_{11}^2, c_{11}^3\in \mathbb{C}, c_{21}^2=c_{12}^2=c_{11}^1,  c_{31}^3=c_{13}^3=c_{11}^1,\\
c_{21}'^2=-c_{12}'^2, c_{21}'^3=-c_{12}'^3, c_{31}'^2=-c_{13}'^2, c_{31}'^3=-c_{13}'^3,\\
c_{32}'^2=-c_{23}'^2, c_{32}'^3=-c_{23}'^3, c_{12}'^2c_{23}'^3=c_{12}'^3c_{23}'^2,\\
c_{13}'^2c_{23}'^3=c_{13}'^3c_{23}'^2, c_{13}'^2c_{12}'^3=c_{12}'^2c_{13}'^3, c_{ij}^{k}=c_{ij}'^{k}=0 \textrm{ otherwise} 
\end{array} \right\}$\\
\hline
$\begin{array}{cccc}
     \mathcal{P}_{3,20} &\not \to&  
\mathcal{P}_{3,4}^{\alpha}(\alpha\neq0),
\mathcal{P}_{3,16}^{\alpha}, \mathcal{P}_{3,17}&  
\end{array}$ &
${\mathcal R}= \left\{ \begin{array}{l}
c_{11}^1, c_{11}^2, c_{11}^3\in \mathbb{C},
c_{21}'^2=-c_{12}'^2, c_{21}'^3=-c_{12}'^3, c_{31}'^2=-c_{13}'^2, \\
c_{31}'^3=-c_{13}'^3, c_{32}'^2=-c_{23}'^2, c_{32}'^3=-c_{23}'^3, c_{12}'^2c_{23}'^3=c_{12}'^3c_{23}'^2,\\
c_{13}'^2c_{23}'^3=c_{13}'^3c_{23}'^2, c_{13}'^2c_{12}'^3=c_{12}'^2c_{13}'^3, c_{ij}^{k}=c_{ij}'^{k}=0 \textrm{ otherwise} 
\end{array} \right\}$\\
\hline
    \end{tabular}
    \caption{Non-degenerations of $3$-dimensional Poisson algebras.}
    \label{tab:algndeg4}
\end{table}

\begin{table}[H]
    \centering
    \begin{tabular}{|rcl|lll|l|}
\hline
\multicolumn{3}{|c|}{\textrm{Degeneration}}  & \multicolumn{3}{|c|}{\textrm{Parametrized change of basis}} & \multicolumn{1}{|c|}{\textrm{Parametrized index}}\\
\hline
\hline
$\mathcal{P}_{3,16}^{*} $ & $\to$ & $\mathcal{P}_{3,15}  $ & $
g_{1}(t)= t^{-1}e_2 ,$ & $
g_{2}(t)= t^{-1}e_1+\frac{1}{2}t^{-2}e_3,$ & $
g_{3}(t)= e_2 + te_3.$ & $f(t) = t^{-1}$\\
\hline

$\mathcal{P}_{3,4}^{*} $ & $\to$ & $\mathcal{P}_{3,3}  $ & $
g_{1}(t)= e_1,$ & $
g_{2}(t)= e_2-(1+t)^{-1}te_3,$ & $
g_{3}(t)= e_2.$ & $f(t) = 1+t$\\
\hline

    \end{tabular}
    \caption{Degenerations of  the $3$-dimensional Poisson algebras for the families $\mathcal{P}_{3,4}^{*}$ and $\mathcal{P}_{3,16}^{*}$.}
    \label{tab:infseriesdeg}
\end{table}

\begin{table}[H]
    \centering
    \begin{tabular}{|rcl|l|}
\hline
\multicolumn{3}{|c|}{\textrm{Non-degeneration}} & \multicolumn{1}{|c|}{\textrm{Arguments}}\\
\hline
\hline

$\mathcal{P}_{3,4}^{*}$ &$\not \to$&  $\mathcal{P}_{3,13}$ 
& 
${\mathcal R}= \left\{ \begin{array}{l}
c_{21}'^{2}=-c_{12}'^{2}, c_{21}'^{3}=-c_{12}'^{3}, c_{31}'^{3}=-c_{13}'^{3},
c_{ij}^{k}=c_{ij}'^{k}=0 \textrm{ otherwise} 
\end{array} \right\}$\\
\hline

$\mathcal{P}_{3,16}^{*}$ &$\not \to$&  
$\mathcal{P}_{3,4}^{\alpha}$, $\mathcal{P}_{3,17}$, $\mathcal{P}_{3,19}$  &
${\mathcal R}= \left\{ \begin{array}{l}
c_{11}^{3}\in \mathbb{C}, c_{21}^{3}=c_{12}^{3}, c_{21}'^{3}=-c_{12}'^{3},
c_{ij}^{k}=c_{ij}'^{k}=0 \textrm{ otherwise} 
\end{array} \right\}$\\
\hline
    \end{tabular}
    \caption{Non-degenerations of  the $3$-dimensional Poisson algebras for the families $\mathcal{P}_{3,4}^{*}$ and $\mathcal{P}_{3,16}^{*}$.}
    \label{tab:infseriesndeg}
\end{table}

\begin{center}
	
	\begin{tikzpicture}[->,>=stealth,shorten >=0.05cm,auto,node distance=1.3cm,
	thick,
	main node/.style={rectangle,draw,fill=black!10,rounded corners=1.5ex,font=\sffamily \scriptsize \bfseries },
	rigid node/.style={rectangle,draw,fill=black!30,rounded corners=1.5ex,font=\sffamily \scriptsize \bfseries }, 
	poisson node/.style={rectangle,draw,fill=black!35,rounded corners=0ex,font=\sffamily \scriptsize \bfseries },
	ac node/.style={rectangle,draw,fill=white!20,rounded corners=0ex,font=\sffamily \scriptsize \bfseries },
	lie node/.style={rectangle,draw,fill=black!30,rounded corners=0ex,font=\sffamily \scriptsize \bfseries },
	style={draw,font=\sffamily \scriptsize \bfseries }]

	\node (2) at (1,4) {$n^2 - n$};
	\node (1) at (1,2) {$n^2-n-1$};
	\node (0)  at (1,0) {$n^2-n-2$};

    \node[ac node] (c21) at (-3,0) {${\mathcal{P}_{1,1}^n}$};
  
    \node[ac node] (c23) at (-4,2) {${\mathcal{P}_{1,3}^n}$};
    \node[lie node] (c24) at (-2,2) {${\mathcal{P}_{1,4}^n}$};

    \node[main node] (c20) at (-3,4) {${\mathcal{P}_{0}^n}$};
    \node[ac node] (c22) at (-5,4) {${\mathcal{P}_{1,2}^n}$};
    \node[lie node] (c25) at (-1,4) {${\mathcal{P}_{1,5}^n}$};
    
	\path[every node/.style={font=\sffamily\small}]

    
    
    
    (c22) edge  (c23)
    (c23) edge  (c21)
    (c25) edge  (c24)
    (c24) edge  (c21)
    (c25) edge  (c23)
    (c20) edge  (c24)
    ;

\end{tikzpicture}
\label{fig3}

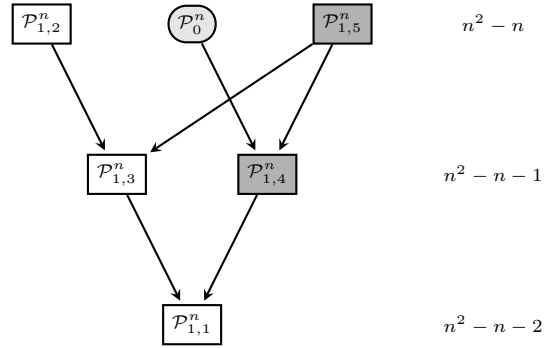
{Figure 3.}  Graph of degenerations and non-degenerations.	
\end{center}


\begin{thebibliography}{}

\bibitem{ale}
Alvarez M.A., 
    On rigid $2$-step nilpotent Lie algebras.
{\em  Algebra Colloq.} {\bf  25 }
   (2018), no. 2, 349--360. 


\bibitem{maria}
Alvarez M.A., Hern\'{a}ndez I., Kaygorodov I.,
    Degenerations of Jordan superalgebras.
{\em  Bull. Malays. Math. Sci. Soc.} {\bf  42}
    (2019), no. 6, 3289--3301.

\bibitem{contr11}
Ancochea Berm\'{u}dez J.M., Fres\'{a}n J., Margalef Bentabol J., 
    Contractions of low-dimensional nilpotent Jordan algebras. 
{\em    Comm. Algebra} {\bf  39} (2011), no. 3, 1139--1151.



\bibitem{bb09}
Beneš T.,  Burde D., 
    Degenerations of pre-Lie algebras. 
 {\em   J. Math. Phys.} {\bf   50}
 (2009), no. 11, 112102, 9pp.
 
\bibitem{bb14}
Beneš T., Burde D.,
    Classification of orbit closures in the variety of three-dimensional Novikov algebras.
{\em    J. Algebra Appl.} {\bf   13}
 (2014), no. 2, 1350081, 33pp.  
 


\bibitem{BC99} 
Burde D., Steinhoff C.,
    Classification of orbit closures of $4$--dimensional complex Lie algebras.
{\em  J. Algebra} {\bf   214} (1999), no. 2, 729--739.



\bibitem{cfk19}
Calderón Martín A.,  Fern\'andez Ouaridi A., Kaygorodov I.,
    The classification of $2$-dimensional rigid algebras.
 {\em  Linear Multilinear Algebra} {\bf 68} (2020), no. 4, 828--844.


\bibitem{degs22}
Fernández Ouaridi A., Kaygorodov I., Khrypchenko M., Volkov Y., Degenerations of nilpotent algebras.
{\em   J. Pure Appl. Algebra} {\bf   226}
  (2022), no. 3, Paper No. 106850, 21pp.

\bibitem{fili}
Filippov V.T. n-Lie algebras. {\em Siberian Math. J.} 26, no. 6, 879–891, (1985). 

\bibitem{g69}
Godbillon C. {\em Géométrie différentielle et mécanique analytique}. Hermann Editeurs. Collection Méthodes. 1969.

\bibitem{gkk19}
Gorshkov I., Kaygorodov I., Khrypchenko M.,
    The geometric classification of nilpotent Tortkara algebras.
  {\em   Comm. Algebra} {\bf  48} (2020), no. 1, 204--209.

\bibitem{gkp}
 Gorshkov I.,  Kaygorodov I.,  Popov Yu., 
   Degenerations of Jordan algebras and ``marginal" algebras. 
{\em    Algebra Colloq.} {\bf  28} (2021), no. 2, 281--294.

\bibitem{gr06}
Goze, M., Remm, E.,
    Non-associative algebras associated to Poisson algebras.
    {\em    J. Algebra} {\bf  320} (2008), no. 294--317.

\bibitem{goze10}
Goze, N. Poisson structures associated with rigid Lie algebras. {\em  Journal of Generalized Lie theory and Applications}. Vol 10. (2010) .


\bibitem{GRH}
Grunewald F.,  O'Halloran J.,
    Varieties of nilpotent Lie algebras of dimension less than six.
{\em   J. Algebra} {\bf   112}
  (1988), no. 2, 315--325. 

\bibitem{GRH2}
Grunewald F., O'Halloran J.,
    A Characterization of orbit closure and applications.
{\em     J. Algebra} {\bf    116} (1988), no. 1, 163--175.


\bibitem{ha17}
Hegazi A., Abdelwahab H.,
    The classification of $N$-dimensional non-associative Jordan algebras with $(N-3)$-dimensional annihilator.
 {\em  Comm. Algebra} {\bf 46} (2018), no. 2, 629--643.

\bibitem{ikv17}
Ismailov N., Kaygorodov I.,  Volkov Yu.,
    The geometric classification of Leibniz algebras.
{\em   Internat. J. Math.} {\bf   29}
   (2018), no. 5, 1850035, 12pp.

\bibitem{ikv19}
Ismailov N., Kaygorodov I.,  Volkov Yu.,
    Degenerations of Leibniz and anticommutative algebras.
{\em   Canad. Math. Bull.} {\bf  62}
   (2019), no. 3, 539--549.

 
\bibitem{S90}
Seeley C., 
    Degenerations of 6-dimensional nilpotent Lie algebras over $\mathbb{C}$. 
 {\em   Comm. Algebra} {\bf  18}
 (1990), no. 10, 3493--3505.

\bibitem{D99}  Dekimpe K., Ongenae V., Filiform left-symmetric algebras,
Geom. Dedicata, 74 (1999), no. 2, 165--199.

\bibitem{K20} Karimjanov IA, Ladra M. Some classes of nilpotent associative
algebras. {\em Mediterranean Journal of Mathematics.} (2020) 17(2):1-21.

\bibitem{K21} Karimjanov, Iqboljon, Ivan Kaygorodov, and Manuel Ladra.
Central extensions of filiform associative algebras. {\em Linear and
Multilinear Algebra} 69, no. 6 (2021): 1083-1101.

\bibitem{kpv}
Kaygorodov I., Popov Yu., Volkov Yu.,
    Degenerations of binary Lie and nilpotent Malcev algebras.
{\em  Comm. Algebra} {\bf  46 } (2018), no. 11, 4928--4940.


\bibitem{kv16}
Kaygorodov I.,   Volkov Yu.,
    The variety of two-dimensional algebras over an algebraically closed field.
{\em  Canad. J. Math.} {\bf   71}
    (2019), no. 4, 819--842.

\bibitem{kv17}
Kaygorodov I., Volkov Yu., 
    Complete classification of algebras of level two.  
{\em  Mosc. Math. J.} {\bf  19 } (2019), no. 3, 485--521.

  
\bibitem{kppv}
Kaygorodov I.,  Popov Yu., Pozhidaev A., Volkov Yu.,
    Degenerations of Zinbiel and nilpotent Leibniz algebras.
{\em   Linear Multilinear Algebra} {\bf  66}
   (2018), no. 4, 704--716.


\bibitem{M14} Masutova, K.K., Omirov, B.A., On some zero-filiform algebras,
{\em Ukrainian Math. J.}, 66 (2014), no. 4,541--552.
  
\bibitem{nambu}
Nambu Y. Generalized Hamiltonian Mechanics.  {\em Phys. Rev.} D7 (1973), 2405–2412.  
  
\bibitem{takh}
Takhtajan L. On foundation of the generalized Nambu mechanics, {\em Comm. Math. Phys.} 160, 295–315, (1994).  
  
\bibitem{Weinstein}
Weinstein, A. 
    Symplectic Geometry. 
    {\em  Bull. Amer. Math. Soc.} {\bf  5 }
    (1981), no. 1-13.   

\end{thebibliography}
\end{document}